\documentclass[12pt]{article}
\RequirePackage[colorlinks,citecolor=blue,urlcolor=blue,linkcolor=blue]{hyperref}
\hypersetup{
colorlinks = true,
citecolor=blue,
urlcolor=blue,
linkcolor=blue,
pdfauthor = {Alexey Kuznetsov, Mateusz Kwasnicki},
pdfkeywords = {stable process, supremum, first passage time, eigenfunctions, semigroup, resolvent, double sine function},
pdftitle = {Spectral analysis of stable processes on the positive half-line},
pdfpagemode = UseNone
}
\usepackage[titletoc,title]{appendix}
\usepackage{graphicx,xspace,colortbl}
\usepackage{amsmath,amsthm,amsfonts}
\usepackage{color}
\usepackage{fancybox}
\usepackage{epsfig}
\usepackage{subfig}
\usepackage{pdfsync}

\newcommand{\alpharho}{}

\newcommand{\ex}[1]{{\mathcal H}_{#1}}
\newcommand{\laplace}{\mathcal{L}}
\newcommand{\gen}{L}
\newcommand{\fourier}{\mathcal{F}}
\newcommand{\mellin}{\mathcal{M}}
\newcommand{\contour}{{\gamma}}

    \def\qed{\hfill$\sqcap\kern-8.0pt\hbox{$\sqcup$}$\\}
    \def\beq{\begin{eqnarray}}
    \def\eeq{\end{eqnarray}}
    \def\beqq{\begin{eqnarray*}}
    \def\eeqq{\end{eqnarray*}}

    \def\re{\textnormal {Re}}
    \def\im{\textnormal {Im}}
    \def\p{{\mathbb P}}
    \def\e{{\mathbb E}}
    \def\r{{\mathbb R}}
    \def\c{{\mathbb C}}
    	
    \def\d{{\textnormal d}}
    \def\i{{\textnormal i}}

    \def\ee{{\textnormal e}}

\newtheorem{theorem}{Theorem}
\newtheorem{lemma}{Lemma}
\newtheorem{proposition}{Proposition}
\newtheorem{corollary}{Corollary}
\theoremstyle{definition}
\newtheorem{definition}{Definition}

\newtheorem{remark}{Remark}

\title{Spectral analysis of stable processes on the positive half-line}
\author{ 
{Alexey Kuznetsov\footnote{Dept. of Mathematics and Statistics,  York University,
4700 Keele Street, Toronto, ON, M3J 1P3, Canada.   Email: kuznetsov@mathstat.yorku.ca}} , \;
{Mateusz Kwa{\'s}nicki\footnote{Institute of Mathematics and Computer Science  Wroc{\l}aw University of Technology  ul. Wybrze{\.z}e Wyspia{\'n}skiego 27  50-370 Wroc{\l}aw, Poland. Email: mateusz.kwasnicki@pwr.edu.pl}}
 }

\date{\today}

\begin{document}
\maketitle

\begin{abstract}
We study the spectral expansion of the semigroup of a general stable process killed on the first exit from the positive half-line. 
Starting with the Wiener-Hopf factorization we obtain the q-resolvent density for the killed process, from which we derive the spectral expansion of the semigroup via the inverse Laplace transform. The eigenfunctions and co-eigenfunctions are given rather explicitly in terms of the double sine function and they give rise to a pair of integral transforms which generalize the classical Fourier sine transform. Our results provide the first explicit example of a spectral expansion of the semigroup of a non-symmetric L\'evy process killed on the first exit form the positive half-line.
\end{abstract}
{\vskip 0.15cm}
 \noindent {\it Keywords}: stable process, first exit time, eigenfunctions, semigroup, resolvent, double sine function
{\vskip 0.25cm}
 \noindent {\it 2010 Mathematics Subject Classification }: Primary 60G52, Secondary 60J35

\section{Introduction}

Consider the Brownian motion process $W=\{W_t\}_{t\ge 0}$ started from $W_0=x>0$ and let $T_0$ be the first time when $W$ hits zero. It is 
well-known 
(see \cite[Appendix 1.3]{Borodin})  that the distribution of $T_0$ can be computed as
\begin{equation}\label{BM1}
 \p_x(T_0>t)=\frac{2}{\pi} \int_0^{\infty} e^{-t \lambda^2} \sin(\lambda x) \lambda^{-1} \d \lambda, 
\end{equation}
and that the transition probability density of the process $W$ killed at $T_0$ has spectral representation
\begin{equation}\label{BM2}
 p_t(x,y)=\frac{2}{\pi} \int_0^{\infty} e^{-t \lambda^2} \sin(\lambda x) \sin(\lambda y) \d \lambda. 
\end{equation}
To give a functional-analytic point of view, let us introduce the semigroup of the killed process as the family of operators
$\{P_t\}_{t\ge 0}$ defined by
\begin{equation}\label{def_semigroup}
P_t u(x)=\e_x[u(W_t) {\mathbf 1}_{\{T_0>t\}}], 
\end{equation}
and let us define the Fourier sine transform as 
\begin{equation}\label{def_Fourier_sine_transform}
\Pi u(\lambda)=\sqrt{2/\pi} \int_0^{\infty} u(x) \sin(\lambda x) \d x, \;\;\; u \in L^1(\r^+).
\end{equation}
It is well-known that $\Pi$ can be extended to an isometry in $L^2(\r^+)$ and then formula \eqref{BM1} gives us a spectral representation of the semigroup in $L^2(\r^+)$: 
\begin{equation}\label{BM3}
P_t = \Pi e^{- t \lambda^2} \Pi. 
\end{equation}
The function $\sin(\lambda x)$ plays an important role in the above expressions. It is the eigenfunction of the one-dimensional Laplace operator $\Delta$ (which is simply the second derivative operator, $\Delta u = u''$), with a Dirichlet boundary condition at zero. Note that $\Delta$ is the infinitesimal generator of the scaled Brownian motion $X_t=\sqrt{2}W_t$.

Brownian motion is a very special process, because it enjoys many useful properties: it is a diffusion process, it is a 
L\'evy process (a process with stationary and independent increments), it is also a self-similar process. The list of other Markov processes
with an explicit spectral representation of the semigroup is quite short.
We mention here the well-known case of one-dimensional diffusion processes \cite{Mandl}, branching processes
\cite{Ogura_1969},
a family of symmetric L\'evy processes obtained as a time change of Brownian motion \cite{Kwasnicki2011} and the recent work of Patie and Savov on non-selfadjoint Markov semigroups \cite{Patie_Savov}.

Our goal in this paper is to generalize the results \eqref{BM1}, \eqref{BM2} and \eqref{BM3} to strictly stable L\'evy processes (which we call simply {\it stable processes}) killed on the first exit 
from $(0,\infty)$. This class consists of L\'evy processes which satisfy  the scaling (or, self-similarity) property:
for any $c>0$ the process $\{cX_{t}\}_{t\ge 0}$ (started from $X_0=0$) has the same distribution as $\{X_{c^{\alpha} t}\}_{t\ge 0}$. Such processes exist when the stability parameter  $\alpha$ belongs to the interval $(0,2]$. When $\alpha=2$ we recover the scaled Brownian motion and for $\alpha<2$ we obtain a two-parameter family of processes with jumps that we will discuss in detail in the next section. 
The scaling property of stable processes is the main reason why they are so popular among researchers and why they appear so frequently in various applications originating in Physics, Chemistry and Biology. Stable processes
(or L\'evy flights, as they are also known in these fields) occur in modelling such
diverse phenomena as fluctuations and transport in plasma, turbulent diffusions, seismic series and earthquakes, signal processing and financial time series (see review article \cite{Dubkov}
for a comprehensive list of applications).

\subsection{Stable processes} 

A L\'evy process is usually defined through {\it the characteristic exponent} $\Psi(z):=-\ln \e_0[\exp(\i z X_1)]$. 
The characteristic exponent of a stable process is given by
\begin{equation}\label{def_Psi_stable}
\Psi(z)=|z|^{\alpha} e^{\pi \i \alpha (1/2-\rho) {\textnormal{sign}}(z)}, \;\;\; z\in \r. 
\end{equation} 
Here the parameters $(\alpha,\rho)$ belong to the following set of admissible parameters
\begin{equation}
{\mathcal A}:=\{\alpha \in (0,1], \; \rho \in (0,1)\} \cup 
\{\alpha \in (1,2], \;  1-1/\alpha< \rho < 1/\alpha  \}.  
\end{equation}
The parameter $\alpha$ is the same one that  appeared in the scaling property discussed above, and it can be shown that  
$\rho=\p_0(X_1>0)$, which explains why $\rho$ is called {\it the positivity parameter}.  
Everywhere in this paper we denote $\hat \rho=1-\rho$, and, more generally we will use the ``hat'' notation to refer to any objects obtained from the dual process $\hat X=-X$. 

Our set of admissible parameters excludes processes with one-sided jumps: subordinators or negative subordinators ($\alpha \in (0,1)$ and $\rho \in \{1,0\}$) and the spectrally-negative and spectrally-positive processes ($\alpha \in (1,2)$ and $\alpha\rho=1$ or $\alpha (1-\rho)=1$, respectively). The first two cases are not interesting, as the processes have monotone paths and the expression for the semigroup of the process on the positive half-line is rather simple, and the case of spectrally-positive/negative processes is covered separately in Section \ref{section_one_sided}.

Let us consider some special cases of stable processes. 
When $\alpha=2$ we necessarily have $\rho=1/2$ and the process $X$ in this case is simply the scaled Brownian motion 
$X_t=\sqrt{2}W_t$. When $\alpha=1$ we can rewrite the characteristic exponent in the form
$
\Psi(z)=\sin(\pi \rho) |z|+\i \cos(\pi \rho) z,
$
so that the process $X$ can be written as 
$
X_t=\sin(\pi \rho) Z_t - \cos(\pi \rho) t,
$
where $Z_t$ is the Cauchy process. A stable process with $\rho=1/2$ is symmetric (has the same distribution as $\hat X$) 
and can be obtained as a subordinated Brownian motion. More precisely, let $S_t$ be an $\alpha/2$-stable subordinator defined by the Laplace transform $\e[\exp( - z S_t)]=\exp(-t z^{\alpha/2})$ and independent of the Brownian motion $W$, 
then the process $\{X_t\}_{t\ge 0}$ has the same distribution as $\{\sqrt{2} W_{S_t}\}_{t\ge 0}$.

In the general case (except when $\alpha \in \{1,2\}$) a stable process is a pure-jump L\'evy process characterized by the density of the L\'evy measure 
$\nu\alpharho(x)$, given by
$$
\nu\alpharho(x)=c |x|^{-1-\alpha} {\bf 1}_{\{x>0\}}+\hat c |x|^{-1-\alpha} {\bf 1}_{\{x<0\}},
$$ 
where we have denoted
$
c=\Gamma(1+\alpha) \sin(\pi \alpha \rho)/\pi$ and $\hat c=\Gamma(1+\alpha) \sin(\pi \alpha (1-\rho))/\pi$. 
The L\'evy measure describes the distribution and the intensity of the jumps of the process, and it is connected to
the characteristic exponent via the L\'evy Khintchine formula 
$$
\Psi\alpharho(z)=-\int_{\r} (e^{\i z x}-1-\i z h(x)) \nu\alpharho(x)\d x,
$$
where $h(x)$ is the cutoff function, which is needed to ensure the convergence of the integral. Using the cutoff function $h(x)\equiv 0$ when 
$\alpha<1$ and $h(x)\equiv x$ when $\alpha>1$ one can check that the above integral representation for $\Psi(z)$ can be evaluated in closed form as given in \eqref{def_Psi_stable}. This construction also works when $\alpha=1$ and $\rho=1/2$ (in this case we can take the cutoff function $h(x)\equiv \sin(x)$). 

A stable process $X$ is also a Markov process, and its dynamics can be described by {\it the infinitesimal generator} $\gen$. This operator is defined for a suitable set of functions $u$ as follows  
$$
\gen\alpharho u(x)=\lim\limits_{t\to 0^+} \frac{1}{t} \big( \e_x[u(X_t)]-u(x)\big). 
$$
As we have mentioned above, when $\alpha = 2$ the infinitesimal generator is simply the one-dimensional Laplace operator, $\gen u = u''$. 
In the symmetric case (that is, when $\rho=1/2$) the infinitesimal generator is the fractional Laplace operator 
$\gen=-(-\Delta)^{\alpha/2}$.
In the general case, it  is a (non-local) integro-differential operator having the following form: 
\begin{align*}
&\gen u(x)=\int_{\r} (u(x+y)-u(x)-u'(x)h(y)) \nu\alpharho(y) \d y,
&& \text{if $\alpha \neq 1$,} \\
&\gen u(x)=-\cos(\pi \rho)u'(x)+\frac{1}{\pi} \sin(\pi \rho)\int_{\r} (u(x+y)-u(x)-u'(x)\sin(y))y^{-2} \d y,
&& \text{if $\alpha = 1$}.
\end{align*}
If we denote by $\fourier$ the Fourier transform operator 
$$
\fourier f(z)=\int_{\r} e^{\i z x } f(x) \d x,
$$
then the infinitesimal generator can be represented in a particularly simple form 
$\gen\alpharho=-\fourier^{-1} \Psi\alpharho(-z) \fourier$ (see \cite{Bertoin}[Proposition 9]).
In other words,  the infinitesimal generator $L$ is a pseudo-differential operator \cite{Taylor_1981} 
and the characteristic exponent $\Psi$ is the symbol of this operator.


\subsection{Main Results}


Consider a stable process $X$ started from $X_0=x>0$ and denote by $T_0$ the first exit time from $(0,\infty)$.  
Let  $\{P_t\}_{t\ge 0}$ be the transition semigroup of the process $X$ killed at time $T_0$: these operators are defined by
\eqref{def_semigroup}, but with $X_t$ instead of $W_t$. Then $P_t$ are sub-Markov operators, 
{and} they are contraction operators on $L^2(\r^+)$. The transition probability density $p_t(x,y)$ is defined as the integral kernel of the operator $P_t$: 
\begin{equation}\label{def_p_t_xy}
\p_x(X_t \in \d y,  T_0>t )=p_t(x,y) \d y. 
\end{equation}  
It is known that the transition density exists when $X$ is a stable process (this follows from Theorem 3 and Example 4 in \cite{chaumont2013}). 

From the analytical point of view, killing the process on the first exit from the positive half-line is equivalent to imposing a Dirichlet boundary condition on the negative half-line. Thus, in order to find the eigenfunctions of the infinitesimal generator we need to solve the equation $\gen\alpharho f(x)=\lambda f(x)$ for $x>0$ with the ``boundary condition'' $f(x)=0$ for $x\le 0$. The problem of computing $v(t,x)=P_t u(x)$ (for a suitable $u$) is equivalent to solving the following parabolic partial integro-differential equation $\partial_t v(t,x)=\gen_x\alpharho v(t,x)$ subject to (i) the ``boundary condition'' $v(t,x)=0$ for $x\le 0$ and (ii) the initial condition $v(0,x)=u(x)$. Finally, the transition probability density $p_t(x, y)$ is simply the fundamental solution to this partial integro-differential equation.

In order to present our results, we need to introduce a certain special function $S_2(z)=S_2(z;\alpha)$, called {\it the double sine function}
\cite{Kurokawa,Koyama2007204}. This function can be defined via two functional equations
\begin{equation}\label{S2_two_functional_eqns}
S_2(z+1)=\frac{S_2(z)}{2\sin(\pi z/\alpha)}, \;\;\; S_2(z+\alpha)=\frac{S_2(z)}{2\sin(\pi z)},
\end{equation}
and the normalizing condition $S_2((1+\alpha)/2)=1$. We collect several equivalent definitions and various properties of the double sine function
in Appendix \ref{AppendixA}, here we only mention the following two facts which will be used  most frequently in this paper:  
\begin{itemize}
\item[(i)] The function $S_2(z)$ is a real meromorphic function having poles at points $z \in \{m\alpha +n \; : \; m, n \in {\mathbb N}\}$;
\item[(ii)] For every $b, c\in \r$ have the following asymptotic result 
\begin{align}\label{S_2_asymptotics}
|S_2(b + \i \alpha \ln(e^{\i c}y)/(2\pi))S_2(b - \i \alpha \ln(e^{\i c}y)/(2\pi))|=
\begin{cases}
y^{1/2+\alpha/2-b}(1+o(1)), \;\;&\textnormal{ as } y \to +\infty, \\
y^{-1/2-\alpha/2+b}(1+o(1)), \;\;&\textnormal{ as } y \to 0^+. 
\end{cases}
\end{align} 
\label{properties_S2_function}
Moreover, the above asymptotic result holds uniformly in $b$ and $c$ on compact subsets of $\r$. 
\end{itemize}

Now we introduce two functions $G$ and $F$, which will play an important role in this paper. For
$(\alpha,\rho) \in {\mathcal A}$ and $x\ge 0$ we define 
\begin{equation}\label{def_g}
G(x):=\int_0^{\infty} e^{-z x} 
 z^{\alpha \rho/2-1/2} |S_2(1+\alpha+\alpha\hat\rho/2 + \i \alpha \ln(z)/(2\pi))|^2 \d z
\end{equation}
and
\begin{equation}\label{def_F_eigenfunction}
F(x):=e^{x \cos(\pi \rho)} \sin(x \sin(\pi \rho)+\pi \rho(1- \alpha \hat \rho)/2)+
\frac{\sqrt{\alpha}}{4\pi} S_2(-\alpha \hat \rho) G(x).
\end{equation}
When it will be needed to stress the dependence on the parameters $\alpha$ and $\rho$, we will write $F(x; \alpha, \rho)$ for $F(x)$ and $G(x; \alpha, \rho)$ for $G(x)$. We define $\hat{F}$ and $\hat{G}$ in a similar way, exchanging the roles of $\rho$ and $\hat{\rho}$ (so that $\hat{F}(x) = F(x; \alpha,\hat{\rho})$).

Note that the integral in \eqref{def_g} converges for $x=0$ (and thus for all $x>0$): this is easy to establish using \eqref{S_2_asymptotics}.
Another important observation is that the function $G$ is completely monotone, that  is $(-1)^n G^{(n)}(x)\ge 0$ for all $n \in {\mathbb Z}^+$ and $x>0$. Finally, we note that the function $F$ is bounded on $(0,\infty)$ if $\rho \ge 1/2$, and it grows exponentially if $\rho<1/2$ (with a similar result for $\hat F$ and $\hat \rho$).

The following theorem is our first main result in this paper: here we generalize formulas \eqref{BM1} and \eqref{BM2} which hold in the Brownian motion case.  

\begin{theorem}\label{thm_main} Let $X$ be a stable process defined by parameters $(\alpha,\rho)\in {\mathcal A}$. 
\begin{itemize}
\item[(i)] If $\alpha>1$ or $\rho\ge 1/2$ then for all $x>0$
\begin{equation}\label{eqn_first_exit_time_spectral}
\p_x(T_0>t)= \frac{\sqrt{\alpha}}{\pi} S_2(\alpha \hat \rho)  \int_0^{\infty} e^{-t\lambda^{\alpha}} F(\lambda x) \lambda^{-1} \d \lambda. 
\end{equation}
\item[(ii)] If $\alpha>1$ or $\rho=1/2$ then for all $t,x,y>0$ we have
\begin{equation}\label{Q_spectral_formula}
p_t(x,y)=\frac{2}{\pi} \int_0^{\infty} e^{-t\lambda^{\alpha}}  F(\lambda x) \hat F(\lambda y) 
\d \lambda.  
\end{equation}
\end{itemize}
\end{theorem} 

In the symmetric case Theorem \ref{thm_main} was established in \cite{Kwasnicki2011} (see Example 6.1), and in the non-symmetric case the eigenfunctions $F(x)$ were computed in \cite{Kwasnicki2013} (in both papers the function $G(x)$ was given in an equivalent integral form). 

Our next goal is to study the transition semigroup $P_t$ and to establish an analogue of identity \eqref{BM3}. In the case of non-symmetric stable processes the situation is bound to be more complicated than in the symmetric case, since one of the functions $F$, $\hat F$ is exponentially increasing. In order to properly define the operators which diagonalize the transition semigroup, we first need to introduce a suitable space of test functions. 

\begin{definition}\label{def_set_X}
Set $\zeta:=\frac{\pi}{2}\min(1,1/\alpha)$. Let us denote by $\ex{\alpha}$ the set of functions $u(x)$ satisfying the following conditions
\begin{itemize}
\item[(i)] $u(x)$ is analytic in the sector $|\arg(x)|<\zeta$ and it takes real values on $(0,\infty)$; 
\item[(ii)] For every $\epsilon \in (0,\zeta)$ there exists $\delta=\delta(\epsilon)>0$  such that 
$|u(x)|=O(|x|^{-\delta |x|})$ as $|x|\to \infty$ and $|u(x)| = O(1)$ as $|x| \to 0$ (uniformly in the sector $|\arg(x)|<\zeta-\epsilon$).  
\end{itemize}
\end{definition}
{We need the above conditions to extend the Fourier--Laplace transform of $u$ to a well-behaved analytic function in the sector $|\arg(x)| < \frac{\pi}{2} + \zeta$. This is done in Lemma~\ref{lemma6}.}

The sets $\ex{\alpha}$ are non-empty: for example, $u(x)=(1+x)^{-x} \in \ex{1}$
and $v(x)=\exp(-x^{\alpha}) \in \ex{\beta}$ for $\beta\ge \alpha>1$ (but not for $\beta<\alpha$ or $\alpha\le 1$). 
It is clear that $\ex{\alpha}\equiv \ex{1}$ for $\alpha \le 1$ and $\ex{\alpha} \subset \ex{\beta}$ for $1\le \alpha<\beta\le 2$. The following properties follow easily from the definition: if $u$ and $v$ belong to $\ex{\alpha}$ then the same is true for functions
\[
\begin{aligned}
{\textnormal{(i)}} \;\;& u(x)v(x); \\
{\textnormal{(ii)}} \;\;& au(x)+bv(x) && \text{for all $a,b \in \r$;} \\
{\textnormal{(iii)}} \;\;& u(ax+b) && \text{for all $a>0$,$b\ge 0$;} \\
{\textnormal{(iv)}} \;\;& x^a e^{-b x^{\alpha}} u(x) && \text{for all $a\ge 0$, $b\ge 0$.} 
\end{aligned}
\]
We also record here the following important property:
\[
\textnormal{(v)} \; \text{The restrictions $u \vert_{\r^+}$ of $u \in \ex{\alpha}$ are dense in $L^2(\r^+)$}.
\]
The above property is easy to prove. Indeed, assuming that there exists $v \in L^2(\r^+)$ which is orthogonal to all $u \in \ex{\alpha}$, we obtain $\int_0^\infty v(x) (1 + x)^{-a x} dx = 0$ for all $a > 0$. Substituting $e^y$ for $(1 + x)^{x}$ and writing $w(y) = v(x) \frac{\d y}{\d x}$, we see that $\int_1^\infty w(y) e^{-a y} dy = 0$. Therefore, $w(y) = 0$ for almost all $y > 1$, and so $v(x) = 0$ for almost all $x > 0$.

Next, we define the following integral operators, which generalize Fourier sine transform  
\eqref{def_Fourier_sine_transform}:
\begin{equation}\label{def_operators}
\Pi u(\lambda)=\sqrt{2/\pi}\int_0^{\infty} F( \lambda x) u(x) \d x, \;\;\; 
\hat \Pi u(\lambda)=\sqrt{2/\pi}\int_0^{\infty} \hat F(\lambda x) u(x) \d x,
\end{equation}
where $u: (0,\infty) \mapsto \c$ is bounded and has compact support. 
We denote by $\hat P_t$ the transition semigroup of the dual process $\hat X$, killed on the first exit from $(0,\infty)$. According to Hunt's switching identity (see \cite[Theorem 5]{Bertoin}), the density of the kernel of the dual semigroup is given by $\hat p_t(x,y)=p_t(y,x)$, thus $\hat P_t$ is simply the adjoint operator of $P_t${, when $P_t$ and $\hat P_t$ are considered as operators acting on $L^2(\r^+)$}.

The next theorem is our second main result.  
\begin{theorem}\label{theorem_2}
Assume that $(\alpha,\rho)\in {\mathcal A}$ and $\rho \ge 1/2$. 
\begin{itemize}
\item[(i)] $\Pi$ can be extended to a bounded self-adjoint operator $\Pi: L^2(\r^+) \mapsto L^2(\r^+)$ and $\hat\Pi: \ex{\alpha} \mapsto  L^2(\r^+)$ is a symmetric operator such that $\hat\Pi \ex{\alpha}$ is dense in $L^2(\r^+)$. For all $u \in \ex{\alpha}$ we have
\begin{equation}\label{eqn_Pi_hat_Pi}
\Pi \hat \Pi u= u.  
\end{equation}
\item[(ii)] For $u \in \ex{\alpha}$ and $t>0$ we have
\begin{equation}\label{spectral_formula_2}
P_t u=  \Pi   e^{-t \lambda^{\alpha}}  \hat\Pi u, 
\end{equation}
and
\begin{equation}\label{spectral_formula_2_dual}
 \Pi \hat P_t \hat\Pi u= e^{-t \lambda^{\alpha}} u. 
\end{equation}
\end{itemize}
\end{theorem}

Note that in the symmetric case $\rho=1/2$ we have $\Pi=\hat \Pi$, and then formula \eqref{eqn_Pi_hat_Pi} implies that $\Pi$ is an isometry on $L^2(\r^+)$. This result was established in greater generality in \cite{Kwasnicki2011}. 

\begin{corollary}\label{corollary1}
Assume that $(\alpha,\rho)\in {\mathcal A}$, $\rho \ge 1/2$ and $\lambda>0$. 
\begin{itemize}
\item[(i)] The functions $u_{\lambda}(x):=F(\lambda x)$ are the eigenfunctions of the semigroup $P_t$, that is
$P_t u_{\lambda}=e^{-t\lambda^{\alpha}} u_{\lambda}$. 
\item[(ii)] The functions $\hat F(\lambda x)$ are the eigenfunctions of the dual semigroup $\hat P_t$, in the sense that 
for all $v\in \ex{\alpha}$ we have $\hat P_t \hat \Pi v=\hat \Pi e^{-t \lambda^{\alpha}} v$.
\end{itemize}
\end{corollary}

{In part~(i) of the above theorem, $u_\lambda$ is a bounded function, so that it is a true eigenfunction of the operator $P_t$ acting on the $L^\infty(\r^+)$ space. If $\rho > 1/2$, then it is easy to see that $u_\lambda(x) = O(x^{-1-\alpha})$ as $x \to \infty$ (by~\eqref{S_2_asymptotics}, \eqref{def_g} and Karamata's theorem), and so $u_\lambda$ is in fact in $L^2(\r^+)$. However, when $\rho = 1/2$, $u_\lambda(x)$ oscillates between $\pm 1$ as $x \to \infty$, and so it is not in $L^2(\r^+)$.}

{On the other hand, unless $\rho = 1/2$, the word \emph{eigenfunction} is used in a rather vague sense in part~(ii) of the theorem. The function $\hat{u}_\lambda(x) = \hat F(\lambda x)$ oscillates as $x \to \infty$ with magnitude that grows exponentially fast, and so there is no obvious way in which $\hat P_t \hat{u}_\lambda$ can be defined.}

{We remark that since $P_t$ and $\hat P_t$ are Markovian, they are contractions on $L^\infty(\r^+)$. By duality, they are also contractions on $L^1(\r^+)$, and interpolating between $L^1(\r^+)$ and $L^\infty(\r^+)$ one easily finds that in fact $P_t$ and $\hat P_t$ are contractions on $L^p(\r^+)$ for $p \in [1, \infty]$. We also remark that the semigroups $\{P_t\}_{t \ge 0}$, $\{\hat P_t\}_{t \ge 0}$ are strongly continuous on $L^p(\r^+)$ if $p \in [1, \infty)$: this can be proved by comparing $P_t$ or $\hat{P}_t$ with transition operators of the non-killed process, for which the corresponding result is standard.}

Our third main result gives Laplace and  Mellin transforms of the eigenfunctions. 
\begin{theorem}\label{thm_F_Laplace_Mellin} Assume that $(\alpha,\rho)\in {\mathcal A}$. 
\begin{itemize}
\item[(i)]
For $z>\max(0,\cos(\pi \rho))$ 
\begin{equation}\label{eqn_F_Laplace}
\int_0^{\infty} e^{-zx} F(x)\d x=\frac{\sqrt{\alpha}}{2} S_2(\alpha \rho) \, z^{-\alpha\hat \rho/2-1/2} 
|S_2(1+\alpha/2+\alpha\hat \rho/2 + \i \alpha \ln(z)/(2\pi))|^2.
\end{equation}
The corresponding result for $\hat F(x)$ can be obtained from the above formula by replacing $\rho \mapsto \hat \rho$. 
\item[(ii)] If  $\rho \ge 1/2$ then for $\re(z) \in (-\alpha \hat \rho, 0)$
\begin{equation}\label{eqn_Mellin_F}
\int_0^{\infty} x^{z-1} F(x) \d x=
\frac{ \Gamma(z)  S_2(z)}{2 S_2(\alpha \hat \rho+z)}. 
\end{equation}
\end{itemize}
\end{theorem}

The two formulas \eqref{eqn_F_Laplace} and \eqref{eqn_Mellin_F} can lead to many useful explicit results about the eigenfunctions. 
For example, when $\alpha$ is irrational, one could use the methods of \cite{Hackmann,HubKuz2011,Kuz2011} and obtain 
complete asymptotic expansions of $F(x)$ as $x\to 0^+$ or $x\to +\infty$, as well as convergent power series representations. 
We leave this investigation for future work, and in this paper we only include the discussion of Doney classes in Section \ref{Section_Doney_classes}: as we will see, in this case all expressions involving the double-sine function reduce to simple finite products. 

The paper is organized as follows. In Section \ref{subsection_21} we review some results from the fluctuation theory of L\'evy processes that will be required later on, and we also outline the plan for proving our main results. In Section  \ref{section_preliminary_Results} we study the Wiener-Hopf factors of stable processes and establish some preliminary results. In  Sections 
 \ref{subsection_proof_Thm3}, \ref{subsection_proof_thm_main} and \ref{section_proof_thm_2}
we prove Theorems \ref{thm_F_Laplace_Mellin}, \ref{thm_main} and \ref{theorem_2}, respectively. In Section \ref{Section_Doney_classes} we simplify all of our formulas in the case when the process belongs to one of Doney classes, 
while the special case of spectrally one-sided processes is treated in Section \ref{section_one_sided}. 
Finally, in Section \ref{section_conclusion} we present some concluding remarks.

\section{Proofs of Theorems \ref{thm_main}, \ref{theorem_2} and \ref{thm_F_Laplace_Mellin}}
\label{section_proofs_thm123}

\subsection{The plan for proving our main results}\label{subsection_21}

Our proofs are based on the Wiener-Hopf factorization, which is a key result in the fluctuation theory of L\'evy processes 
\cite{Bertoin, Kyprianou}. We present this result here for the sake of completeness.  
We start with a L\'evy process $X$ (a one-dimensional process with stationary and independent increments) 
and we denote by 
$$
{\overline  X}_{t}=\sup\{ X_s \, :\, 0 \le s \le t\}, \;\;\; 
{\underline X}_t=\inf\{ X_s \, : \,  0 \le s \le t\}
$$
the running supremum/infimum processes. We will denote by $\ee(q)$ 
an exponential random variable with expected value $1/q$, and assume that $\ee(q)$ is independent of the process $X$. The following result is contained in Theorem 6.15(i) and identity (6.28) in \cite{Kyprianou}:  
\begin{theorem}[The Wiener-Hopf factorization]\label{theorem_WHF}
Let $X$ be a L\'evy process started from zero. Then for $q>0$
\begin{itemize}
\item[(i)] the random variable ${\overline X}_{\ee(q)}$ is independent of $X_{\ee(q)}-{\overline X}_{\ee(q)}$;
\item[(ii)] the random variables $X_{\ee(q)}-{\overline X}_{\ee(q)}$ and ${\underline X}_{\ee(q)}$ have the same distribution. 
\end{itemize}
\end{theorem}

Let us explain how  we will use the above Wiener-Hopf factorization result in order to obtain information about the transition probability density of the process killed on the first exit from $(0,\infty)$. 
Assume that $X_0=0$ and denote by $f_{\overline X}(x)$ and $f_{\underline X}(x)$ the probability density functions of 
${\overline X}_{\ee(1)}$ and $-{\underline X}_{\ee(1)}$, respectively (we will prove later that these densities exist 
and express them in terms of the double sine function $S_2$).  Due to the scaling property of stable processes, the random variable ${\overline X}_{\ee(q)}$ has density $q^{1/\alpha} f_{\overline X}(xq^{1/\alpha})$,  with a similar result for $\underline X_{\ee(q)}$. Let us denote by $H_q(x,y,z)$ the joint density of 
$(X_{\ee(q)}, \underline X_{\ee(q)})$ for the process started at $X_0=x$:
$$
\p_x(X_{\ee(q)} \in \d y, \underline X_{\ee(q)} \in \d z )=H_q(x,y,z)\d y \d z,
\;\; {\textnormal{ where }} \; z< \min(x,y). 
$$
According to the Theorem \ref{theorem_WHF}, the pair $(X_{\ee(q)},\underline X_{\ee(q)})$ 
has the same distribution as $(x+S_q-I_q,x-I_q)$, where $S_q$ and $I_q$ are {\it independent} random variables having the same distribution as 
$\overline X_{\ee(q)}-x$ and $x-\underline X_{\ee(q)}$, respectively. Therefore the function $H_q(x,y,z)$ can be written as follows
\begin{equation}\label{eqn_H_q_factorization}
H_q(x,y,z)=q^{2/\alpha} f_{\overline X}((y-z)q^{1/\alpha}) f_{\underline X}((x-z)q^{1/\alpha}). 
\end{equation}
The above identity is well-known in the theory of fluctuations of L\'evy processes. For example, see the proof of Theorem 18 in 
\cite{Doney2007}
 
Let us denote by $h_t(x,y,z)$ the joint density of $(X_t,\underline X_t)$ for the process started from $X_0=x$, that is  
$$
\p_x(X_t \in \d y, {\underline X}_t \in \d z)= h_t(x,y,z) \d y \d z,
\;\; {\textnormal{ where }} \; 0<z< \min(x,y). 
$$
This function is related to the transition probability density and the semigroup of the killed process via the following identities
\begin{equation}\label{eqn_p_t_P_t}
p_t(x,y)=\int_0^{\min(x,y)} h_t(x,y,z)\d z, \;\;\;
P_t u(x)=\int_0^{x} \left[ \int_z^{\infty} h_t(x,y,z) u(y) \d y \right] \d z,
\end{equation}
where we have assumed that $u$ is bounded on $(0,\infty)$. 
At the same time, by conditioning on the random variable $\ee(q)$ (which has exponential distribution with the density
$qe^{-qt}{\mathbf 1}_{\{t>0\}}$) we see that 
\begin{equation}\label{eqn_Ht_Hq}
\int_0^{\infty} qe^{-q t} h_t(x,y,z)\d t=H_q(x,y,z), 
\end{equation}
and therefore the function $q \mapsto q^{-1}H_q(x,y,z)$ is the Laplace transform of $t\mapsto h_t(x,y,z)$. Our plan for proving Theorem \ref{thm_main}(ii) and \ref{theorem_2}(ii) is to invert the Laplace transform in \eqref{eqn_Ht_Hq} and to obtain
\begin{equation}\label{eqn_H_inv_Laplace}
h_t(x,y,z)=\frac{1}{2\pi \i} \int_{\i \r} q^{2/\alpha-1} f_{\overline X}((y-z)q^{1/\alpha}) f_{\underline X}((x-z)q^{1/\alpha}) e^{qt} \d q,
\end{equation}
with $f_{\overline X}(x)$ and $f_{\underline X}(x)$ expressed in terms of the double sine function, and then deform the contour of integration, so that the vertical line $\i \r$ is transformed into Hankel's contour (beginning at $-\infty$, going around $0$ in counter-clockwise direction and ending at $-\infty$, see Lemma \ref{lemma_2} below).
As we will see, a lot of effort is required for justifying this transformation of the contour of integration.

\begin{remark}
The function $H_q(x,y,z)$ is closely related to {\it resolvent operators}, defined as 
$$
R_qu(x)=\int_0^{\infty} e^{-qt} P_t u(x) \d t=q^{-1} \e_x[ u(X_{\ee(q)}) {\mathbf 1}_{\{\underline X_{\ee(q)}>0\}}]. 
$$
It is clear from \eqref{eqn_p_t_P_t} that $R_q$ is an integral operator with the kernel
$$
r_q(x,y)=q^{-1} \int_0^{\min(x,y)} H_q(x,y,z)\d z. 
$$
\end{remark}

\subsection{Some preliminary results}\label{section_preliminary_Results}

In the next theorem we identify the Wiener-Hopf factors for a stable process $X$, which are defined as the Laplace transform of positive random variables ${\overline X}_{\ee(1)}$ and  $-{\underline X}_{\ee(1)}$:
\begin{equation}
\phi(z)=\e [e^{-z {\overline X}_{\ee(1)}}], \;\;\; \hat \phi(z)=\e[e^{z {\underline  X}_{\ee(1)}}], \;\;\; \re(z)\ge 0.  
\end{equation}

\begin{theorem}\label{thm_phi_S2} 
Assume that $(\alpha,\rho)\in {\mathcal A}$. 
For  $\re(z)\ge 0$
  \begin{align}\label{eq_phi_S2}
 \phi(z)=&z^{-\alpha \rho/2} S_2(1/2 +\alpha/2+\alpha\rho/2 + \i \alpha\ln(z)/(2\pi))\\
  & \;\;\;\; \;\times S_2(1/2 +\alpha/2+\alpha\rho/2 - \i \alpha\ln(z)/(2\pi)), \nonumber
 \end{align}
and $\hat \phi(z)$ can be obtained from the above formula by replacing $\rho \mapsto \hat \rho$.  
\end{theorem}

The above result was established in \cite{Kuz2011} in the case when $\alpha \neq 1$ (see Theorem 4 in \cite{Kuz2011} and formula 
\eqref{eqn_S2_Barnes_Gamma} below). Here we present a much simpler proof which also covers the case $\alpha=1$. 

\vspace{0.2cm}
\noindent
{\bf Proof of Theorem \ref{thm_phi_S2}:} 
The proof is based on the Wiener-Hopf factorization result, Theorem \ref{theorem_WHF}. Writing 
$X_{\ee(1)}={\overline X}_{\ee(1)}+(X_{\ee(1)}-{\overline X}_{\ee(1)})$ and using properties (i) and (ii) we obtain the following factorization
\begin{equation}\label{WH_factorization}
\frac{1}{1+\Psi\alpharho(z)}=\e[e^{\i zX_{\ee(1)}}]=\e[e^{\i z{\overline X}_{\ee(1)}}] \times 
\e[e^{\i z {\underline X}_{\ee(1)}}]=\phi(-\i z) \times \hat \phi(\i z), \;\;\; z\in \r. 
\end{equation}
This is a classical Riemann-Hilbert problem: we need to find a function $\phi(-\i z)$  analytic in the upper half-plane $\im(z)>0$ and continuous in the closed upper half-plane and another function $\hat \phi(\i z)$ having the same properties but in the lower half-plane  $\im(z)<0$ which meet at the boundary $\im(z)=0$ as prescribed by \eqref{WH_factorization}. 

Let us define by $f(z)$ the function in the right-hand side of \eqref{eq_phi_S2}, and by $\hat{f}(z)$ the same function, but with $\rho$ replaced by $\hat{\rho}$. First we will verify that 
the functions $f$ and $\hat{f}$ satisfy  \eqref{WH_factorization}. 
Assume that $z>0$ and let us denote $w=\alpha \ln(z)/(2\pi \i)$. Then
\begin{align*}
f(-\i z) \hat{f}(\i z)  &=
z^{-\alpha \rho/2} e^{\pi \i \alpha \rho/4} S_2(1/2+\alpha/4+\alpha\rho/2+w) S_2(1/2+3\alpha/4+\alpha \rho/2-w) \\
&\times 
z^{-\alpha (1-\rho)/2} e^{-\pi \i \alpha (1-\rho)/4} S_2(1/2+5\alpha/4-\alpha\rho/2+w) S_2(1/2+3\alpha/4-\alpha \rho/2-w)\\
&=z^{-\alpha/2} e^{\pi \i \alpha (\rho-1/2)/2}
S_2(1/2+3\alpha/4+\alpha \rho/2-w)S_2(1/2+5\alpha/4-\alpha\rho/2+w),
\end{align*}
where we have used identity \eqref{S2_reflection_formula} in the form
$$
 S_2(1/2+\alpha/4+\alpha\rho/2+w)S_2(1/2+3\alpha/4-\alpha \rho/2-w)=1.
$$
Next, according to the second functional equation  in \eqref{S2_two_functional_eqns} we have
$$
S_2(1/2+5\alpha/4-\alpha\rho/2+w)=\frac{S_2(1/2+\alpha/4-\alpha\rho/2+w)}{2 \sin(\pi(1/2+\alpha/4-\alpha\rho/2+w))}.
$$
Using the above result and identity \eqref{S2_reflection_formula} in the form
$$
S_2(1/2+\alpha/4-\alpha\rho/2+w)S_2(1/2+3\alpha/4+\alpha \rho/2-w)=1
$$
we obtain
\begin{align*}
f(-\i z) \hat{f}(\i z) &=\frac{z^{-\alpha/2} e^{\pi \i \alpha (\rho-1/2)/2}}
{2 \sin(\pi(1/2+\alpha/4-\alpha\rho/2+w))}\\
&=\frac{z^{-\alpha/2} e^{\pi \i \alpha (\rho-1/2)/2}}
{z^{\alpha/2} e^{\pi \i \alpha (1/2-\rho)/2}+z^{-\alpha/2} e^{\pi \i \alpha (\rho-1/2)/2}}=\frac{1}{1+z^{\alpha} e^{\pi \i \alpha (1/2-\rho)}}=\frac{1}{1+\Psi\alpharho(z)}. 
\end{align*}
Thus we have verified that that the functions $f$ and $\hat{f}$ satisfy \eqref{WH_factorization} for $z>0$, and the proof for $z<0$ follows by taking the complex conjugate in the above equation.

To prove that our candidate solutions $f$ and $\hat{f}$ are in fact the Wiener-Hopf factors, we need to apply a certain uniqueness argument. Using the fact that $S_2(z)$ is analytic and zero-free in the strip
$0<\re(\alpha)<1+\alpha$ we check that the functions $f$ and $\hat{f}$ are analytic and zero-free in the half-plane $\re(z)\ge 0$. The asymptotic result \eqref{S_2_asymptotics} easily gives us $z^{-1} \ln(f(z)) \to 0$ as $z\to \infty$ uniformly in the half-plane $\re(z)\ge 0$ (with a similar result for $\hat{f}$). To finish the proof we only need to apply the uniqueness result given in \cite[Theorem 1(f)]{Kuz2011b} and conclude that $\phi \equiv f$. 
\qed

\begin{proposition}\label{thm_Xe1_CM}
Assume that $(\alpha,\rho)\in {\mathcal A}$. The distribution of the random variable ${\overline X}_{\ee(1)}$ is a mixture of exponentials and its density is given by 
\begin{equation}\label{px_completely_monotone}
f_{\overline X}(x)=\int_0^{\infty} e^{- x u} \mu(u) \d u, \;\;\; x>0
\end{equation}
where
\begin{equation}\label{def_mu}
\mu(u):=
 \frac{1}{\pi} \sin(\pi \alpha \rho)
  u^{\alpha \hat \rho/2} | S_2(1/2+\alpha+\alpha\rho/2 + \i \alpha \ln(u)/(2\pi))|^2.
\end{equation}
The function $f_{\underline X}(x)$ (the density of $-\underline X_{\ee(1)}$) can be obtained from above equations by replacing $\rho \mapsto \hat \rho$. 
\end{proposition} 
\begin{proof}
 According to 
\cite[Theorem 2]{Rogers83}, the random variable ${\overline X}_{\ee(1)}$ is a mixture of exponentials, therefore we can write its density in the form \eqref{px_completely_monotone} with some positive measure $\mu(\d u)$. Applying Fubini's Theorem we see that  
$$
\phi(z)=\int_0^{\infty} e^{-zx} f_{\overline X}(x) \d x=\int_0^{\infty} \frac{\mu(\d u)}{u+z},
$$
thus the Wiener-Hopf factor $\phi(z)$ given by \eqref{eq_phi_S2}  is the Stieltjes transform of the measure $\mu(\d u)$. 
Using the well-known result on the inversion of Stieltjes transform we conclude that the measure $\mu(\d u)$ has a density
$\mu(u)$ which can be found via
\begin{equation}\label{eqn_mu_lambda}
\mu(u)=-\frac{1}{\pi} \im[ \phi(e^{\pi \i} u)]. 
\end{equation}
From formula \eqref{eq_phi_S2} we find (as before, denoting $w=\i \alpha \ln(u)/(2\pi)$)
\begin{equation*}
\phi(e^{\pi \i} u)=u^{-\alpha \rho/2} e^{- \pi \i \alpha \rho/2} S_2(1/2+\alpha+\alpha \rho/2+w)
S_2(1/2+\alpha\rho/2-w).
\end{equation*}
Applying the second functional equation in \eqref{S2_two_functional_eqns} we check that
\begin{align*}
S_2(1/2+\alpha\rho/2-w)&=S_2(1/2+\alpha+\alpha\rho/2-w) 2 \sin(\pi (1/2+\alpha\rho/2-w))\\
&=S_2(1/2+\alpha+\alpha\rho/2-w) \left(e^{\pi \i \alpha \rho/2} u^{-\alpha/2}+e^{-\pi \i \alpha \rho/2} u^{\alpha/2}\right).
\end{align*}
Combining the above two formulas we finally obtain
$$
\phi(e^{\pi \i} u)=u^{-\alpha \rho/2} |S_2(1/2+\alpha+\alpha \rho/2+\i \alpha \ln(u)/(2\pi))|^2
\left(u^{-\alpha/2}+e^{-\pi \i \alpha \rho} u^{\alpha/2}\right),
$$
and now the desired result \eqref{def_mu} follows easily from \eqref{eqn_mu_lambda} and the above equation. 
\end{proof}

Let us denote by ${\mathcal R}$ the Riemann surface of the logarithm function. In what follows we will often consider 
functions defined on ${\mathcal R}$ or on sectors in ${\mathcal R}$. In particular, we have $\ln(e^{\i c} z)=\i c + \ln(z)$ for all 
$c\in \r$ and $z\in {\mathcal R}$.  Next, we state two simple lemmas, which will be used often in this paper. Both of these results are well-known and can be easily established by Cauchy Residue Theorem.   
\begin{lemma}[Rotating the contour of integration]\label{lemma_1}
Assume that the function $f(z)$, $z\in {\mathcal R}$, is analytic in the sector $-\epsilon<\arg(z)<b+\epsilon$ for some $b>0$ and $\epsilon>0$, except for a finite number of poles at points $z=z_i$ lying in the sector $0<\arg(z)<b$. Assume also that for some  $\delta>0$  we have $f(z)=O(|z|^{-1+\delta})$ as $|z|\to 0^+$ and $f(z)=O(|z|^{-1-\delta})$ as $|z|\to +\infty$, uniformly in the sector $0\le \arg(z) \le b$. 
Then 
$$
\int_0^{\infty} f(z) \d z=e^{\i b}\int_0^{\infty} f(e^{\i b} z)\d z+
2\pi \i \sum_i {\textnormal{Res}}(f(z) \,:\, z=z_i). 
$$
\end{lemma}

\begin{lemma}[Hankel's contour of integration]\label{lemma_2}
Assume that the function $f(z)$, $z\in {\mathcal R}$ satisfies $\overline{f(z)}=f(\bar z)$ and is analytic in the domain
 $\pi/2-\epsilon<|\arg(z)|<\pi+\epsilon$ for some $\epsilon>0$, except for a finite number of poles at $z=z_i$ in the sector $\pi/2<\arg(z)<\pi$ and the corresponding poles $z=\bar z_i$ in the sector $-\pi<\arg(z)<-\pi/2$.  Assume also that for some  $\delta>0$  we have
$f(z)=O(|z|^{-1+\delta})$ as $|z|\to 0^+$ and $f(z)=O(|z|^{-1-\delta})$ as $|z|\to +\infty$, uniformly in the sector 
 $\pi/2 \le \arg(z)\le \pi$. 
Then
$$
\frac{1}{2\pi \i}\int_{\i \r} f(z) \d z=-\frac{1}{\pi} \int_0^{\infty} \im[f(e^{\pi \i}z)] \d z+
2 \re \Big[ \sum_i {\textnormal{Res}}(f(z) \,:\, z=z_i) \Big] . 
$$
\end{lemma}

The following elementary result will be useful in deriving various estimates. 
\begin{lemma}\label{lemma_Mellin_bound}
Let $f_1(x):=\min(x^{a},x^b)$ and $f_2(x):=\min(x^c, x^d)$, where we assume that $a\ge b$, $c\ge d$, $a+c+1>0$ and $b+d+1<0$. 
Define $\xi:=\min(a+1,-d)$ and $\eta:=\max(b+1,-c)$. 
Then $\xi\ge \eta$ and there exists $C=C(a,b,c,d)>0$ such that 
$\int_0^{\infty} f_1(x) f_2(x/\lambda) \d x \le
C \min(\lambda^{\xi},\lambda^{\eta})$ for all $\lambda>0$.
\end{lemma}
The proof of the above result is very simple -- one only needs to evaluate the integral $\int_0^{\infty} f_1(x) f_2(x/\lambda) \d x$, distinguishing the two cases $\lambda>1$ and $\lambda<1$. We leave all the details to the reader.

\subsection{Proof of Theorem \ref{thm_F_Laplace_Mellin}}\label{subsection_proof_Thm3}

\vspace{0.25cm}
\noindent
{\bf Proof of Theorem \ref{thm_F_Laplace_Mellin}(i):}
We denote
\begin{align}\label{def_f_new}
f(z):=\frac{\sqrt{\alpha}}{2}S_2(\alpha\rho) \, &z^{-\alpha\hat \rho/2-1/2} 
S_2(1+\alpha/2+\alpha\hat \rho/2 +  \i \alpha \ln(z)/(2\pi))\\
& \qquad  \;\;\;\; \times S_2(1+\alpha/2+\alpha\hat \rho/2 -  \i \alpha \ln(z)/(2\pi)). \nonumber
\end{align} 
We consider $f(z)$ as a function on the Riemann surface ${\mathcal R}$. 
{Part~(i) of Theorem~\ref{thm_F_Laplace_Mellin} states that
$$
\int_0^\infty e^{-z x} F(x) dx = f(z)
$$
for $z > \max(0, \cos(\pi \rho))$. We}
will establish an equivalent statement 
\begin{equation}\label{eqn_F_inverse_Laplace}
F(x)=\frac{1}{2\pi \i} \int_{c+\i \r} 
f(z) e^{zx} \d z,
\end{equation}
where $c>\max(0,\cos(\pi \rho))$. 

Let us denote $\eta=2 \pi \min(1,1 / \alpha)+\rho \pi$ (it is easy to check that $\eta>\pi$ for all $(\alpha,\rho)\in {\mathcal A}$). 
Given the fact that the double sine function $S_2(z)$ has poles at points $\{m+n\alpha \, : \, m,n\ge 1\}$ and the pole at $1+\alpha$ is simple (see Appendix \ref{AppendixA}) we see that the function $f(z)$ is analytic in the sector 
$|\arg(z)|< \eta$, except for two simple poles at points $z_{\pm}=\exp(\pm \pi \i \rho)$. 
Let us compute the residues at these poles. From the first functional equation in \eqref{S2_two_functional_eqns} we find that
$$
S_2(1+\alpha/2+\alpha\hat \rho/2-\i \alpha \ln(z)/(2\pi))=\frac{S_2(\alpha/2+\alpha\hat\rho/2-\i\alpha \ln(z)/(2\pi))}
{2\sin(\pi(1/2+\hat\rho/2-\i\ln(z)/(2\pi))}.
$$
We also check that
$$
\frac{\d}{\d z} 2\sin(\pi(1/2+\hat\rho/2-\i\ln(z)/(2\pi)) \Big \vert_{z=z_+}=\frac{\i}{z_+}.
$$
Combining the above two formulas with \eqref{def_f_new} we obtain
\begin{align*}
{\textnormal{Res}}(f  :  z=z_+)&=-\i \frac{\sqrt{\alpha}}{2}S_2(\alpha\rho)
\exp(-\pi \i \rho (1+\alpha\hat \rho)/2+\pi \i \rho) S_2(\alpha)S_2(1+\alpha \hat \rho)\\
&=-\frac{\i}{2} \exp(\pi \i \rho(1-\alpha \hat \rho)/2),
\end{align*}
where we have also used formulas \eqref{S2_special_values} and \eqref{S2_reflection_formula}. 

Note that 
\begin{equation}\label{f_bound}
|f(z)| \le C(\alpha,\rho) \times \min(1,|z|^{-\alpha \hat \rho-1}), 
\end{equation}
as $|z|\to \infty$ or $|z| \to 0$ uniformly in the sector $|\arg z| < \eta$ (this upper bound follows from \eqref{S_2_asymptotics} and 
\eqref{def_f_new}). In particular, the integral in~\eqref{eqn_F_inverse_Laplace} converges, and we can shift the contour of integration in \eqref{eqn_F_inverse_Laplace} so that $c+\i \r \mapsto \i \r$ and then pass to Hankel's contour of integration (see Lemma \ref{lemma_2}) and obtain
\begin{align} \nonumber
&\frac{1}{2\pi \i} \int_{c+\i \r} 
f(z) e^{zx} \d x=
2 \re\left[ - \frac{\i}{2} \exp(e^{\pi \i \rho}x+\pi \i \rho(1-\alpha \hat \rho)/2)) \right]
- \frac{1}{\pi} \int_{0}^{\infty} e^{-zx} \im[f(ze^{\pi \i})] \d z
  \\
&= e^{\cos(\pi \rho) x} 
\sin(x \sin(\pi \rho)+\pi \rho(1-\alpha \hat \rho)/2)-\frac{1}{\pi} \int_{0}^{\infty} e^{-zx} \im[f(ze^{\pi \i})] \d z .
\label{thm2i_proof0}
\end{align}
Using the definition of $f(z)$ in \eqref{def_f_new} we check that
\begin{align*}
f(ze^{\pi \i})=\frac{\sqrt{\alpha}}{2} S_2(\alpha\rho) e^{-\pi \i (1+\alpha \hat \rho)/2}z^{-\alpha \hat \rho/2-1/2} S_2(1+\alpha+\alpha\hat\rho/2-\i \alpha \ln(z)/(2\pi))
S_2(1+\alpha\hat\rho/2+\i\alpha \ln(z)/(2\pi)).
\end{align*}
From the second functional equation in \eqref{S2_two_functional_eqns} it follows that
$$
S_2(1+\alpha\hat\rho/2+\i \alpha \ln(z)/(2\pi))=
\i\left[ e^{\pi \i \alpha \hat\rho/2}z^{-\alpha/2}-e^{-\pi \i \alpha \hat\rho/2} z^{\alpha/2} \right] 
S_2(1+\alpha+\alpha\hat\rho/2+\i \alpha \ln(z)/(2\pi)).
$$
The above two equations give us 
$$
\im[f(ze^{\pi \i})]=\frac{\sqrt{\alpha}}{2}
\sin(\pi \alpha \hat \rho)
S_2(\alpha \rho) 
z^{\alpha \rho/2-1/2} |S_2(1+\alpha+\alpha\rho/2+  \i \alpha \ln(z)/(2\pi))|^2.
$$
Applying the 
the second functional equation in \eqref{S2_two_functional_eqns}  we check that 
$$
\frac{\sqrt{\alpha}}{2} \sin(\pi \alpha \hat \rho)S_2(\alpha \rho) =
-\frac{\sqrt{\alpha}}{4} S_2(-\alpha \hat \rho). 
$$
Combining the above two results with \eqref{thm2i_proof0} we obtain formula \eqref{eqn_F_inverse_Laplace}. 
\qed

\begin{lemma}\label{lemma_3} Assume that $(\alpha,\rho)\in {\mathcal A}$. Then $F(x)=O(x^{\alpha \hat \rho})$ as 
$x\to 0^+$. 
\end{lemma}
\begin{proof}
Let us consider the function $G(x)$ defined by \eqref{def_g}. Using
\eqref{S_2_asymptotics} we check that the integrand in \eqref{def_g} satisfies 
\begin{equation}\label{eqn_g_est}
 z^{(\alpha \rho-1)/2} |S_2(1+\alpha+\alpha\hat \rho/2 + \i \alpha \ln(z)/(2\pi))|^2 
 = 
 \begin{cases}
 z^{\alpha}(1+o(1)), \;\; &{\textnormal{ as $z\to 0^+$}}, \\
 z^{-\alpha \rho-1}(1+o(1)), \;\; &{\textnormal{ as $z\to +\infty$}}. 
 \end{cases}
\end{equation}
This result combined with  \eqref{def_g} implies that $G(0^+)<+\infty$ and $G'(0^+)=-\infty$. We conclude that 
the function $F(x)$ is monotone in some interval $(0,\epsilon)$. 
Using  \eqref{S_2_asymptotics} we check that the function $f(z)$ (the Laplace transform of $F(x)$ given by 
\eqref{eqn_F_Laplace}) satisfies 
$f(z)= C z^{-\alpha \hat \rho-1} (1+o(1))$ as $z\to +\infty$,
for some constant $C=C(\alpha,\rho)>0$. Applying Karamata's Tauberian theorem followed by the Monotone Density Theorem we conclude that 
$F(x)= C x^{\alpha \hat \rho} (1+o(1))$ as $x\to 0^+$.  
\end{proof}

\vspace{0.25cm}
\noindent
{\bf Proof of Theorem \ref{thm_F_Laplace_Mellin}(ii):}
Let us again denote by $f(z)$ the Laplace transform of $F(x)$, given by \eqref{eqn_F_Laplace}. From formulas 
\eqref{S_2_asymptotics} and \eqref{eqn_F_Laplace} we check that the upper bound 
\eqref{f_bound} is true for all $z \in (0,\infty)$. At the same time, 
when $\rho\ge 1/2$ the function $F(x)$ is bounded  and according to Lemma \ref{lemma_3} it satisfies 
$F(x)=O(x^{\alpha\hat \rho})$ as $x\to 0^+$. Therefore we can apply Fubini's Theorem and conclude that
$$
\int_0^{\infty} f(z) z^{-s} \d z =\Gamma(1-s) \int_0^{\infty} x^{s-1} F(x) \d x,
$$
where both integrals converge absolutely for $\re(s) \in (-\alpha \hat\rho, 0)$. 
The integral identity \eqref{eqn_tau_binomial} implies 
\begin{equation}\label{formula_Mellin_f}
\int_0^{\infty} z^{-s} f(z) \d z=   \frac{\pi }{S_2(1-s)S_2(\alpha\hat \rho+s)}, \;\;\; -\alpha \hat\rho<\re(s)<1. 
\end{equation}
Formula \eqref{eqn_Mellin_F} follows from the above two equations by using \eqref{S2_reflection_formula}, the second functional equation in \eqref{S2_two_functional_eqns} and applying the reflection formula for the gamma function. 
\qed

\subsection{Proof of Theorem \ref{thm_main}}\label{subsection_proof_thm_main}

In the next lemma we collect some properties of the  function 
$$
\mu(u)=
 \frac{1}{\pi} \sin(\pi \alpha \rho)
  u^{\alpha \hat \rho/2} | S_2(1/2+\alpha+\alpha\rho/2 + \i \alpha \ln(u)/(2\pi))|^2,
$$
which has first appeared in Proposition \ref{thm_Xe1_CM}.

\begin{lemma}\label{lemma_4}
Assume that $(\alpha,\rho)\in {\mathcal A}$. 
\begin{itemize}
\item[(i)] The function $\mu(u)$ is analytic in the sector 
$|\arg(u)|<\pi(1/\alpha+\hat \rho)$, $u\in {\mathcal R}$, except for two simple poles at points $u_{\pm}=\exp(\pm \pi \i (1/\alpha-\rho))$. The residues at these poles are given by
\begin{equation}\label{residue_mu}
{\textnormal{Res}}(\mu(u) \, : \, u=u_{\pm})=\frac{S_2(\alpha \rho)}{2\pi \sqrt{\alpha}}
e^{\mp \pi \i(\alpha \rho \hat \rho/2+3 \rho/2-1/\alpha)}.
\end{equation} 
\item[(ii)] Denote $C:=\sin(\pi \alpha \rho)/\pi$. Then
  \begin{align}\label{asymptotics_mu}
  |\mu(u)|=
\begin{cases}
|u|^{\alpha}(C+o(1)), \;\;\;&\textnormal{ as } |u| \to 0^+, \\
|u|^{-\alpha \rho}(C+o(1)), \;\;\;&\textnormal{ as } |u| \to +\infty, 
\end{cases}
  \end{align}
uniformly in the sector $|\arg(u)|<\pi (1/\alpha + \hat \rho)$. 
\end{itemize}
\end{lemma}
\begin{proof}
Let us rewrite
 the expression in \eqref{def_mu} in the form
\begin{equation}\label{def_mu2}
\mu(u)= C u^{\alpha \hat \rho/2} 
S_2(1/2+\alpha+\alpha\rho/2 + \i \alpha \ln(u)/(2\pi))
S_2(1/2+\alpha+\alpha\rho/2 - \i \alpha \ln(u)/(2\pi)).
\end{equation}
The double sine function $S_2(z)$ is a meromorphic function which has poles at points 
$\{m+\alpha n \, : \, m,n\ge 1\}$, and the pole at $z=1+\alpha$ is simple (see Appendix \ref{AppendixA}). This implies that the function $\mu(u)$ is analytic in the sector $|\arg(u)|<\eta := \pi/\alpha-\pi\rho+2\pi\min(1,1/\alpha)$, except for the two simple poles at $u=u_{\pm}$. It is easy to see that $\eta \ge \pi/\alpha + \pi \hat{\rho}$. Let us compute the residue at $u=u_-$. 
In this case the function $S_2(1/2+\alpha+\alpha\rho/2-\i \alpha \ln(u)/(2\pi))$ is analytic in the neighborhood of $u=u_-$ and the pole comes from the other factor $S_2(1/2+\alpha+\alpha\rho/2+ \i \alpha \ln(u)/(2\pi))$, which we transform with the help of 
\eqref{S2_two_functional_eqns} as follows
$$
S_2(1/2+\alpha+\alpha\rho/2+\i \alpha \ln(u)/(2\pi))=\frac{S_2(1/2+\alpha\rho/2+\i \alpha \ln(u)/(2\pi))}
{2\sin(\pi(1/2+\alpha\rho/2+\i \alpha \ln(u)/(2\pi))}.
$$
Note that 
$$
\frac{\d}{\d u} \left[ 2\sin(\pi(1/2+\alpha\rho/2+\i \alpha \ln(u)/(2\pi)) \right] \; \Big \vert_{u=u_-}=
\frac{\alpha}{\i u_-} ,
$$
thus 
\begin{align*}
{\textnormal{Res}}(\mu(u) \, : \, u=u_-)&=
\frac{\sin (\pi \alpha \rho)}{\pi} \times 
\frac{\i u_-}{\alpha} \times  (u_-)^{\alpha \hat\rho/2} 
S_2(1/2+\alpha+\alpha\rho/2-\i \alpha \ln(u_-)/(2\pi))\\
&\qquad \qquad \qquad \qquad \qquad \;\;\;\;\;\;\times
S_2(1/2+\alpha\rho/2+\i \alpha \ln(u_-)/(2\pi))\\
&=\frac{\sin(\pi \alpha \rho)}{\pi \alpha}
e^{\pi \i/2-\pi \i (1/\alpha-\rho)(1+\alpha \hat \rho/2)} S_2(\alpha+\alpha \rho) S_2(1)\\
&=\frac{S_2(\alpha \rho)}{2\pi \sqrt{\alpha}}
e^{\pi \i(\alpha \rho \hat \rho/2+3 \rho/2-1/\alpha)},
\end{align*}
where in the last step we have used the functional equation \eqref{S2_two_functional_eqns} and formula \eqref{S2_special_values}. Thus we have proved  
\eqref{residue_mu} for $u=u_-$, and the result for $u=u_+$ follows by taking the complex conjugate. 

The result in item (ii) follows at once from  \eqref{S_2_asymptotics} and \eqref{def_mu2}. 
\end{proof}

\begin{lemma}\label{lemma_5} Assume that $(\alpha,\rho)\in {\mathcal A}$. 
\begin{itemize}
\item[(i)] The function $f_{\overline X}(x)$ can be extended to an analytic function in the sector 
$|\arg(x)|<\pi(1/\alpha+\hat \rho)$, $x\in {\mathcal R}$.
\item[(ii)] There exists a constant $C=C(\alpha,\rho)$  such that for all $x \in {\mathcal R}$ in the sector $|\arg(x)|<\pi(1/\alpha+\hat \rho)$ we have
\begin{align}\label{f_X_upper_bound}
|f_{\overline X}(x)|<
\begin{cases}
C \min(|x|^{\alpha \rho-1}, |x|^{-\alpha-1}), &  {\textnormal{ if }} |\arg(x)|< \pi (1/\alpha-\rho), \\
C \min(|x|^{\alpha \rho-1}, |x|^{-\alpha-1})+C e^{- \xi |x|}, \; &{\textnormal{ if }} |\arg(x)|\ge \pi (1/\alpha-\rho),
\end{cases}
\end{align}  
where we have denoted $\xi:=\cos(|\arg(x)|-\pi(1/\alpha-\rho))$. 
\item[(iii)] For $x>0$ 
\begin{equation}\label{p_piialpha_F}
 e^{\pi \i/\alpha} f_{\overline X}(e^{\pi \i /\alpha} x)=
\frac{2}{\sqrt{\alpha}} S_2(1+\alpha \rho)
\left(\hat F(x)+e^{\pi \i \rho}\hat F'(x)\right). 
\end{equation}
\end{itemize}
The corresponding results for $f_{\underline X}(x)$ can be obtained by changing $\rho \mapsto \hat \rho$ and $\hat F\mapsto F$. 
\end{lemma}
\begin{proof}
Let us first prove part (i). 
We start with the equation \eqref{px_completely_monotone}, which defines $f_{\overline X}(x)$ as an analytic function in the right half-plane $\re(x)>0$. Assume that $x$ lies in the first quadrant, that is $\arg(x) \in (0,\pi/2)$. Choose any $\beta \in (0,\pi/2)$ such that $\beta\neq \pi(1/\alpha-\rho)$. Applying Lemmas \ref{lemma_1} and \ref{lemma_4}, we rotate the contour of integration in \eqref{px_completely_monotone} by angle $\beta$ in the clockwise direction, 
so that $\r^+ \mapsto e^{-\i \beta} \r^+$. Taking into account the pole of $\mu(u)$ at $u=u_-$ (which will lie in the sector $-\beta < \arg(u)< 0$ if $\beta>\pi(1/\alpha-\rho)$) we obtain the following identity
\begin{equation}\label{f_x_rotation_of_contour}
f_{\overline X}(x)=e^{-\i \beta} \int_0^{\infty} e^{-e^{-\i \beta} z x } \mu(e^{- \i \beta} z)\d z -
2 \pi \i \times  {\textnormal{Res}}(\mu(u) \, : \, u=u_-) e^{-(u_-) x} {\bf 1}_{\{\beta > \pi(1/\alpha -\rho)\}}. 
\end{equation}
 Note that the integral in \eqref{f_x_rotation_of_contour} converges if $\re(e^{-\i \beta} x)>0$, thus we have obtained an analytic continuation of $f(x)$ into the half-plane $\arg(x) \in (\beta-\pi/2, \beta+\pi/2)$. If $\beta + \pi/2< \pi(1/\alpha+\hat \rho)$, we can repeat the above procedure: choose $x$ in the sector $\arg(x) \in (\beta, \beta+\pi/2)$ and rotate the contour of integration by an angle $\beta$ clockwise. Eventually we will cover the whole sector $\arg(x) \in (-\pi/2, 1/\alpha+\hat \rho)$. Finally, we extend $f_{\overline X}(x)$ into the sector $\arg(x) \in (-1/\alpha-\hat \rho,\pi/2)$ by the symmetry principle (the conjugate of  
${f_{\overline X}(x)}$ is $f_{\overline X}(\bar x)$).  This ends the proof of part (i).

Let us now prove part (ii). Everywhere in this proof we will denote by $A_i$ some positive constants which can depend only on $(\alpha,\rho)$. We choose any $\beta \in (0,\pi(1/\alpha-\rho-1/8))$ (if this interval is empty we take $\beta=0$). We set $x=e^{ \i \beta} y$ in \eqref{f_x_rotation_of_contour} and rewrite that equation in the form  
\begin{equation}\label{f_x_rotation_of_contour2}
f_{\overline X}(e^{\i \beta} y)=e^{- \i \beta} \int_0^{\infty} e^{- z y } \mu(e^{- \i \beta} z)\d z.  
\end{equation}
Note that $\beta < \pi( 1/\alpha-\rho)$, so the second term in \eqref{f_x_rotation_of_contour} vanishes. 
Formula \eqref{f_x_rotation_of_contour2} is valid in the half-plane $\re(y)>0$. The function $\mu(u)$ is analytic in the sector 
$|\arg(u)|<\pi (1/\alpha-\rho)$, which includes the sector $|\arg(u)|\le \pi(1/\alpha-\rho-1/8)$. Given the behavior of $\mu(u)$ as $u\to 0^+$ or 
$u\to +\infty$ given in  \eqref{asymptotics_mu}, we see that there must exist a constant $A_1$ such that 
\begin{equation}\label{mu_bound}
|\mu(u)|<A_1 \min(|u|^{\alpha},|u|^{-\alpha \rho}), 
\end{equation}
for all $u$ in the sector $|\arg(u)|\le \pi (1/\alpha-\rho-1/8)$. 
Note that $\re(y)>|y|/2$ in the sector $|\arg(y)|<\pi/4$, thus for all $y$ in the sector $|\arg(y)|<\pi/4$ we have
\begin{equation}\label{f_upper_bound}
|f_{\overline X}(e^{\pi \i \beta} y)| < A_1 \int_0^{\infty} e^{- z \times |y|/2}  \min(z^{\alpha},z^{-\alpha \rho}) \d z.
\end{equation}
 We leave it as an exercise to check that the above estimate implies that there exists a constant $A_2$ such that
$$
|f_{\overline X}(e^{\pi \i \beta} y)|<A_2 \min(|y|^{\alpha\rho-1}, |y|^{-\alpha-1}), 
$$ 
for $|\arg(y)|<\pi/4$. 
Since the sector $ |\arg(x)| < \pi(1/\alpha-\rho+1/8)$ can be covered by finitely many sectors of angle $\pi/2$, we have proved that there exists some constant $A_3$ such that for all $x$ in the sector 
$|\arg(x)|<\pi (1/\alpha-\rho+1/8)$ we have 
$$
|f_{\overline X}(x)|<A_3 \min(|x|^{\alpha\rho-1}, |x|^{-\alpha-1}). 
$$

\begin{figure}
\centering
\includegraphics[height =7cm]{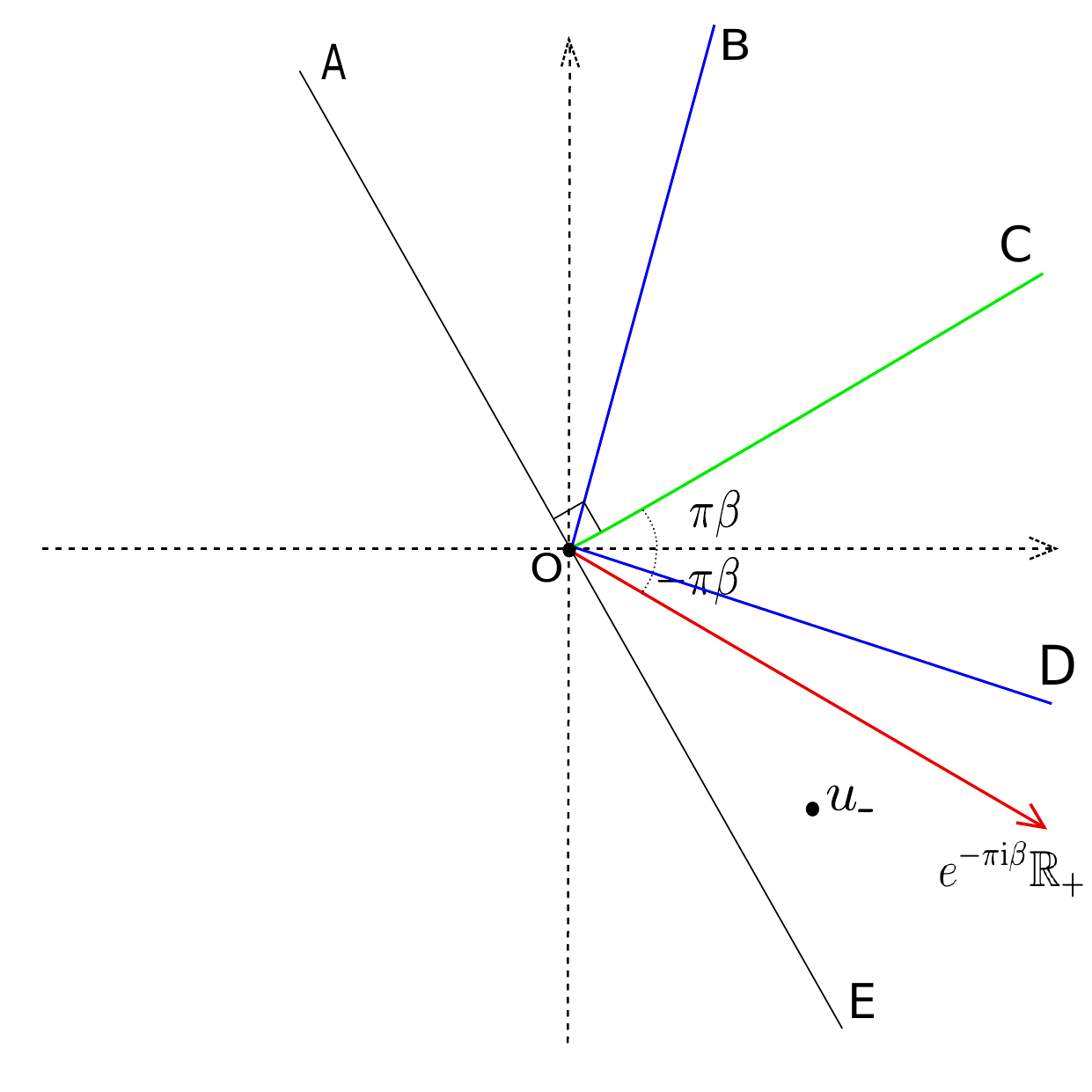}
\caption{Illustration to the proof of Lemma \ref{lemma_5}.}
\label{fig_1}
\end{figure}

Next, we take any $\beta \in (\pi(1/\alpha-\rho+1/8),\pi(1/\alpha-\rho+1))$ and repeat the same procedure. We set $x=e^{\i \beta} y$ in \eqref{f_x_rotation_of_contour} and rewrite that equation in the form  
\begin{equation}\label{f_x_rotation_of_contour3}
f_{\overline X}(e^{\i \beta} y)=e^{-\i \beta} \int_0^{\infty} e^{- z y } \mu(e^{-\i \beta} z)\d z
-2 \pi \i \times  {\textnormal{Res}}(\mu(u) \, : \, u=u_-) e^{-(u_-) e^{\i \beta} y}. 
\end{equation}
The integral term is estimated as above, and the exponential term is estimated as 
$$
|e^{-(u_-) e^{\i \beta} y}|=\exp(-\re(e^{-\pi \i (1/\alpha-\rho)+ \i \beta+\i \arg(y)}) |y|)
=\exp(-\cos(\beta-\pi (1/\alpha-\rho)) |y|). 
$$
This ends the proof of item (ii). 

The ideas for the above proof of item (ii) are illustrated in Figure \ref{fig_1}. The contour of integration in \eqref{f_x_rotation_of_contour}  is over the red line $e^{-\pi \i \beta} \r^+$. With this contour of integration, the integral representation \eqref{f_x_rotation_of_contour} is valid in the half-plane AOE, and we obtain asymptotic estimates of 
$f_{\overline X}(x)$ in the sector BOD. The asymptotic estimates can be made uniform, as long as the contour of integration
$e^{-\pi \i \beta} \r^+$ does not pass too close to the pole of $\mu(u)$ as $u=u_-$.

Let us now prove part (iii). Setting $\beta = 1/\alpha$ in the equation \eqref{f_x_rotation_of_contour3} and using 
formula \eqref{residue_mu} we obtain
\begin{equation}\label{p_rotate_eqn1}
e^{\pi \i/\alpha} f_{\overline X}(e^{\pi \i /\alpha} x)=
- \i \frac{S_2(\alpha \rho)}{\sqrt{\alpha}}\exp(-xe^{\pi \i \rho} + \pi \i \alpha \rho \hat \rho/2+3\pi \i \rho/2)
+\int_0^{\infty} e^{-z x} \mu(e^{-\pi \i/\alpha}z) \d z.
\end{equation}
Next, using \eqref{def_mu} we compute 
\begin{equation*}
\mu(e^{-\pi \i/\alpha} z)=
\frac{1}{\pi} \sin(\pi \alpha \rho) z^{\alpha \hat \rho/2} e^{-\pi \i \hat \rho/2}
S_2(\alpha+\alpha \rho/2 -\i \alpha \ln(z)/(2\pi ))S_2(1+\alpha+\alpha \rho/2+\i \alpha\ln(z)/(2\pi)).
\end{equation*}
The functional equation \eqref{S2_two_functional_eqns} gives us the identity
\begin{align*}
S_2(\alpha+\alpha \rho/2-\i \alpha \ln(z)/(2\pi))&=2 \sin(\pi(1+\rho/2-\i  \ln(z)/(2\pi))) S_2(1+\alpha+\alpha \rho/2-\i \alpha \ln(z)/(2\pi))\\
&=\i  \left(e^{\pi \i \rho/2} z^{1/2}-e^{-\pi \i \rho/2} z^{-1/2}\right)
S_2(1+\alpha+\alpha \rho/2-\i \alpha \ln(z)/(2\pi))
\end{align*}
and simplifying the result we obtain
\begin{equation*}
\mu(e^{-\pi \i/\alpha} z)=\frac{1}{\pi} \sin(\pi \alpha \rho)
z^{\alpha \hat \rho/2-1/2}(e^{\pi \i \rho} z-1) |S_2(1+\alpha+\alpha \rho/2+\i \alpha \ln(z)/(2\pi \i))|^2 .
\end{equation*}
Combining the above result with \eqref{p_rotate_eqn1} and \eqref{def_g} we get 
\begin{align*}
e^{\pi \i/\alpha} f_{\overline X}(e^{\pi \i /\alpha} x)&=
- \i \frac{S_2(\alpha \rho)}{\sqrt{\alpha}}\exp(-xe^{\pi \i \rho} + \pi \i \alpha \rho \hat \rho/2+3\pi \i \rho/2)-
\frac{1}{\pi} \sin(\pi \alpha \rho) (\hat G(x)+e^{\pi \i \rho} \hat G'(x))
\\
&=\frac{2}{\sqrt{\alpha}} S_2(1+\alpha \rho)
\times \Big[-\i \sin(\pi \rho)\exp(-xe^{\pi \i \rho} + \pi \i \alpha \rho \hat \rho/2+3\pi \i \rho/2)
\\
&\qquad\qquad\qquad\qquad\;\;\;\;\;-\frac{\sqrt{\alpha} \sin(\pi \alpha \rho)}{2\pi S_2(1+\alpha \rho)}(\hat G(x)+e^{\pi \i \rho} \hat G'(x))  \Big].
\end{align*}
Using formulas \eqref{S2_two_functional_eqns} and \eqref{S2_reflection_formula} we check that 
$$
-\frac{\sqrt{\alpha} \sin(\pi \alpha \rho)}{2\pi S_2(1+\alpha \rho)}=\frac{\sqrt{\alpha}}{4\pi} S_2(-\alpha \rho),
$$
and we leave it as an exercise to verify that 
$$
-\i \sin(\pi \rho)\exp(-xe^{\pi \i \rho} + \pi \i \alpha \rho \hat \rho/2+3\pi \i \rho/2)=\hat{K}(x)+e^{\pi \i \rho} \hat{K}'(x),
$$
where 
$$
\hat{K}(x):=e^{-x \cos(\pi \rho)} \sin(x \sin(\pi \rho)+\pi \hat \rho(1- \alpha \rho)/2). 
$$
The above four formulas combined with the definition of $\hat F(x)$ (see formula \eqref{def_F_eigenfunction}) imply the desired identity \eqref{p_piialpha_F}.
\end{proof}

\begin{remark}
The proof of part (i) of Lemma \ref{lemma_5} shows that $f_{\overline X}(x)$ can in fact be extended to an analytic function on the whole of 
${\mathcal R}$, though this result will not be used anywhere in this paper. Also, note that \eqref{p_piialpha_F} implies the following results,
\begin{align*}
\frac{2}{\sqrt{\alpha}} S_2(1+\alpha \rho) \hat F(x)&=e^{\pi \i (1/\alpha-\rho)}  
\int_0^x \exp(e^{-\pi \i \rho} (y-x)) f_{\overline X}(e^{\pi \i/\alpha} y) \d y\\
&=-\frac{1}{\sin(\pi \rho)}\im\bigl(e^{\pi \i/\alpha - \pi \i \rho} f_{\overline X}(e^{\pi \i /\alpha} x)\bigr),
\end{align*}
which establish a direct link between the spectral theory of stable processes
on the half-line and the Wiener-Hopf factorization theory. 
\end{remark}

Let us consider what happens to the asymptotic behavior of the function $f_{\overline X}(x)$ (as $|x| \to +\infty$) as we increase $|\arg(x)|$. Lemma  \ref{lemma_5}(ii) tells us 
that there is a transition from power-type decay $|q|^{-\alpha-1}$ to exponential growth 
$\exp(-\xi |q|)$ which occurs at the critical level $|\arg(x)|=\sigma:=\pi (1/\alpha-\rho+1/2)$. For values of $|\arg(x)|<\sigma$
we have $\xi>0$, thus we have a power-type decay; precisely at the critical level 
$|\arg(x)|=\sigma$ we have $\xi=0$ which results in a bounded oscillatory behavior; finally, when $|\arg(x)|$ exceeds the critical level $\sigma$ we have $\xi \in (-1,0)$ 
and $f_{\overline X}(x)$ exponentially increases and oscillates. It is easy to check that for all $(\alpha,\rho) \in {\mathcal A}$ we have
$\sigma>\pi/(2\alpha)$. Using this result we rewrite the upper bound in Lemma  \ref{lemma_5}(ii) in the following less informative (but more useful) form:
\begin{corollary}\label{corollary2}
 There exist constants $C=C(\alpha,\rho)>0$ and $\theta=\theta(\alpha,\rho)>0$ 
such that for all $q$ in the sector $|\arg(q)|<\pi/2+\theta$ we have
\begin{align}\label{f_X_upper_bound1}
|f_{\overline X}(q^{1/\alpha})|\le C \min(|q|^{\rho-1/\alpha}, |q|^{-1-1/\alpha}),
\end{align}  
and for all $q$ in the sector $\pi/2+\theta<|\arg(q)|\le \pi$ we have
\begin{align}\label{f_X_upper_bound2}
|f_{\overline X}(q^{1/\alpha})| \le 
\begin{cases}
C \min(|q|^{\rho-1/\alpha}, |q|^{-1-1/\alpha}), \; &  {\textnormal{ if }} \rho < 1/2, \\
C (|q|^{\rho-1/\alpha}+1), \; &  {\textnormal{ if }} \rho = 1/2, \\
C (|q|^{\rho-1/\alpha}+\exp(|q|^{1/\alpha})) \; &{\textnormal{ if }} \rho>1/2. 
\end{cases}
\end{align} 
The corresponding results for $f_{\underline X}(x)$ can be obtained by changing $\rho \mapsto \hat \rho$.
\end{corollary}

The next corollary will play a crucial role in the proof of Theorem \ref{thm_main}(ii) and Theorem \ref{theorem_2}(ii). 
\begin{corollary}\label{corollary_lucky_integral}
Assume that $(\alpha,\rho)\in {\mathcal A}$. Then for all $x,y>0$  we have 
\begin{equation}\label{lucky_integral}
\int_0^{\min(x,y)}
\im \left[ e^{2\pi \i/\alpha}  f_{\underline X}((x-z)  e^{\pi \i/\alpha} )
f_{\overline X}((y-z) e^{\pi \i/\alpha})\right]\d z=
\frac{2}{\alpha}
 F(x)\hat F(y). 
\end{equation}
\end{corollary}
\begin{proof}
First of all, we use formulas  \eqref{S2_two_functional_eqns} and \eqref{S2_reflection_formula} to check that 
$$
S_2(1+\alpha \rho) S_2(1+\alpha \hat \rho)=\frac{1}{2\sin(\pi \rho)}. 
$$
Next, we use the identity \eqref{p_piialpha_F} and compute
\begin{align*}
&\im \left[ e^{2\pi \i/\alpha} f_{\underline X}((x-z)  e^{\pi \i/\alpha})
f_{\overline X}((y-z) e^{\pi \i/\alpha}) \right] \\ 
&=
\frac{2}{\alpha \sin(\pi \rho)}\im \left[ 
( F(x-z)+e^{\pi \i \hat \rho}  F'(x-z)) (\hat F(y-v)+e^{\pi \i \rho} \hat F'(y-z))\right]\\
&=
\frac{2}{\alpha}\left[ F(x-z) \hat F'(y-z)
+ F'(x-z) \hat F(y-z)\right]=- \frac{2}{\alpha} \frac{\d}{\d z} \left[ F(x-z) \hat F(y-x)\right] .
\end{align*}
To finish the proof one needs to integrate the above expression in $z\in (0,\min(x,y))$ and use the fact that 
$F(0+)=\hat F(0+)=0$ (see Lemma \ref{lemma_3}). 
\end{proof}

\vspace{0.25cm}
\noindent
{\bf Proof of Theorem \ref{thm_main}, part (i):}
Let us denote by $m_t(y)$ the probability density of the random variable $-\underline X_t$ when $X_0=0$ (the existence of this density is known - see \cite{chaumont2013,doney2010,Kuz2011}, though it also follows easily from the absolute convergence of the integral in formula \eqref{eqn_h_t_Laplace} below). 
As we have discussed in Section \ref{subsection_21}, the scaling property of stable processes implies that the probability density
function of $-{\underline X}_{\ee(q)}$ is given by $q^{1/\alpha-1} f_{\underline X}(yq^{1/\alpha})$. Conditioning on the exponential random variable $\ee(q)$ we arrive at the Laplace transform identity  
$$
\int_0^{\infty} e^{-q t} m_t(y) \d t=q^{1/\alpha-1} f_{\underline X}(yq^{1/\alpha}). 
$$
We write down $m_t(y)$ as the inverse Laplace transform 
and then  pass to Hankel's contour (via Lemma \ref{lemma_2} and Corollary \ref{corollary2}; here we need $\alpha > 1$ or $\hat \rho \le 1/2$):
\begin{equation}\label{eqn_h_t_Laplace}
m_t(y)=\frac{1}{2\pi \i} \int_{\i \r}  q^{1/\alpha-1} f_{\underline X}(yq^{1/\alpha})e^{qt} \d q
=\frac{1}{\pi} \int_0^{\infty} e^{-q t} q^{1/\alpha-1} \im\left[ e^{\pi \i /\alpha} f_{\underline X}(y q^{1/\alpha} e^{\pi \i/\alpha}) \right]\d q.
\end{equation}
Next we compute
\begin{align*}
\p_x(T_0>t )&=\p_0( - \underline X_t<x)= \int_0^x m_t(y) \d y\\
&=
\frac{1}{\pi} \int_0^{\infty} e^{-q t} q^{1/\alpha-1}
\left\{
\int_0^x
\im\left[ e^{\pi \i /\alpha} f_{\underline X}(y q^{1/\alpha} e^{\pi \i/\alpha}) \right]\d y \right\}\d q, 
\end{align*}
where the application of Fubini's Theorem is justified due to Corollary \ref{corollary2}. 
The integral in curly brackets can be evaluated using formula \eqref{p_piialpha_F}: 
\begin{align*}
\int_0^x
\im\left[ e^{\pi \i /\alpha} f_{\underline X}(y q^{1/\alpha} e^{\pi \i/\alpha}) \right]\d y=
\frac{2}{\sqrt{\alpha}} S_2(1+\alpha \hat \rho) \sin(\pi \rho) 
\int_0^x
  F'(y q^{1/\alpha}) \d y=
 \frac{S_2(\alpha \hat \rho)}{\sqrt{\alpha}}  q^{-1/\alpha} F(x q^{1/\alpha}), 
\end{align*}
where we have used the fact that $F(0+)=0$ (Lemma \ref{lemma_3}) and identity \eqref{S2_two_functional_eqns}. 
Formula \eqref{eqn_first_exit_time_spectral} follows by combining the above two identities and changing the variable of integration
$q=\lambda^{\alpha}$. 
\qed

\vspace{0.25cm}
\noindent
{\bf Proof of Theorem \ref{thm_main}, part (ii):}
We recall our notation from Section \ref{subsection_21}: the function $h_t(x,y,z)$ denotes the joint probability density of 
random variables $(X_t, \underline X_t)$ for a stable process $X$ started from $X_0=x$. 
We start with equation \eqref{eqn_H_inv_Laplace}, which we reproduce here for convenience: 
\begin{equation}\label{eqn_H_inv_Laplace2}
h_t(x,y,z)=\frac{1}{2\pi \i} \int_{\i \r} q^{2/\alpha-1} f_{\overline X}((y-z)q^{1/\alpha}) f_{\underline X}((x-z)q^{1/\alpha}) e^{qt} \d q. 
\end{equation}
According to Corollary \ref{corollary2}, for fixed $0<z<\min(x,y)$ the integrand is bounded from above by $C \min(1,|q|^{-3})$ when $q\in \i \r$, thus the integral in \eqref{eqn_H_inv_Laplace2} converges absolutely. 

Now we plan to pass to Hankel's contour of integration in \eqref{eqn_H_inv_Laplace2}. Corollary \ref{corollary2} tells us that 
in some sector $\pi/2\le |\arg(q)|<\pi/2+\theta$ the integrand is bounded from above by $C \min(1,|q|^{-3})$. 
In the sector $\pi/2+\theta\le |\arg(q)|\le \pi$ the term $e^{qt}$ is exponentially decaying. When $\alpha>1$, the integrand is bounded by $C e^{2|q|^{1/\alpha}+qt}$, which decays to zero exponentially. When $\alpha\le 1$ and $\rho=1/2$ the integrand is bounded by $C e^{qt}$, which again decays to zero exponentially. Thus all conditions of Lemma \ref{lemma_2} are satisfied and 
we obtain
\begin{equation*}\label{eqn_H_inv_Laplace3}
h_t(x,y,z):=\frac{1}{\pi} \int_{0}^{\infty}  
\im \left[ e^{2\pi \i/\alpha} f_{\overline X}((y-z) e^{\pi \i/\alpha} q^{1/\alpha}) f_{\underline X}((x-z)  e^{\pi \i/\alpha} q^{1/\alpha})
\right]  q^{2/\alpha-1} e^{-qt} \d q. 
\end{equation*}
Assume first that $x\neq y$. We use formula \eqref{eqn_p_t_P_t} and apply Fubini's Theorem to obtain
\begin{align}\label{eqn_thm1_proof3}
p_t(x,y)&=\int_0^{\min(x,y)} h_t(x,y,z)\d z\\
\nonumber
&=
\frac{1}{\pi} \int_{0}^{\infty}  
\left\{ \int_0^{\min(x,y)} \im \left[ e^{2\pi \i/\alpha} f_{\overline X}((y-z) e^{\pi \i/\alpha} q^{1/\alpha}) f_{\underline X}((x-z)  e^{\pi \i/\alpha} q^{1/\alpha})
\right] \d z \right\}  q^{2/\alpha-1} e^{-qt} \d q. 
\end{align}
To justify the application of Fubini's Theorem, we use Corollary \ref{corollary2} and estimate the integrand as follows:
\begin{align}\label{eqn_thm1_proof4}
&|f_{\overline X}((y-z) e^{\pi \i/\alpha} q^{1/\alpha}) f_{\underline X}((x-z)  e^{\pi \i/\alpha} q^{1/\alpha})|
  q^{2/\alpha-1} e^{-qt}\\ \nonumber
  &<C q^{2/\alpha-1} e^{-qt} ((y-z)^{\alpha \rho-1} q^{\hat\rho-1/\alpha}+A(q))
  ((x-z)^{\alpha\hat \rho-1}q^{\rho-1/\alpha}+B(q)),
\end{align}
where $A(q)=B(q)=1$ if $\rho=1/2$ and $\alpha\le 1$ and $A(q)=B(q)=\exp(\max(x,y) |q|^{1/\alpha})$ if $\alpha>1$. Considering these two cases separately ($\alpha\le 1$ and $\alpha>1$) one can check that the function in the right-hand side of \eqref{eqn_thm1_proof4} is integrable in $(z,q) \in \{ 0<z< \min(x,y), \; q>0\} \subset \r^2$ as long as $x\neq y$.

To finish the proof, we compute the integral in curly brackets in 
\eqref{eqn_thm1_proof3} using Corollary \ref{corollary_lucky_integral} and change the variable of integration $q=\lambda^{\alpha}$.
To remove the restriction $x\neq y$ we note that the right-hand side in  \eqref{Q_spectral_formula} defines a continuous function of $t,x,y>0$, and we define $p_t(x,x)$ by continuity. 
\qed

\subsection{Proof of Theorem \ref{theorem_2}}\label{section_proof_thm_2}

Recall that we have defined $\zeta:=\frac{\pi}{2}  \min(1,1/\alpha)$. Let us define ${\mathcal S}_{r}$ to be the ``shift operator": ${\mathcal S}_r u(x)=u(r+x)$, and let $\laplace u(z):=\int_0^{\infty} u(x) e^{-xz} \d x$ denote the Laplace transform of $u$. Note that $u \in \ex{\alpha}$ implies ${\mathcal S}_r u \in \ex{\alpha}$ for all $r\ge 0$.

\begin{lemma}\label{lemma6}
Let $u\in \ex{\alpha}$. Then $\laplace {\mathcal S}_r u(z)$ is an entire function 
and for every small $\epsilon>0$ there exists $C=C(u,\alpha,\epsilon)$ such that for all $r\ge 0$ 
we have $|\laplace  {\mathcal S}_r u(z)| \le C \min(1,|z|^{-1})$ in the sector $|\arg(z)|\le \zeta+\pi/2-\epsilon$.   
\end{lemma}

\begin{proof}
The fact that $\laplace {\mathcal S}_r u(z)$ is an entire function follows from Definition \ref{def_set_X}, which implies that $u(x)$ is bounded for small $x=0$ and decays faster than $\exp(-\delta x\ln(x))$ for large $x$. 

Let us fix $\beta \in (-\zeta,\zeta)$. By the same estimate as above, we can rotate the contour of integration and change the variable of integration $x=e^{\i \beta} y$ to obtain
\begin{equation}\label{lem6_eqn1}
\laplace {\mathcal S}_r u(z)=e^{\i \beta} \int_0^{\infty} e^{-y e^{\i \beta} z} u(r +  e^{\i \beta} y) \d y.
\end{equation}
According to Definition \ref{def_set_X}, there exists $\delta(u, \alpha, \epsilon)>0$ and $C(u, \alpha, \epsilon)>0$ 
such that $|u(r + e^{\i \beta} y)|\le C  \min(1, y^{-\delta y})$ for all $y > 0$. Using the estimate $ \re(e^{\i \beta} z) \ge 
 \sin(\epsilon)|z|$ in the sector $|\arg(e^{\i \beta}z)|\le \pi/2-\epsilon$ we obtain (for all $z$ in this sector)
$$
|\laplace {\mathcal S}_r u(z)| \le C \int_0^\infty e^{-\sin(\epsilon) |z|y }\d y \le \frac{C}{\sin(\epsilon) |z|} \, . 
$$
Furthermore,
$$
|\laplace {\mathcal S}_r u(z)| \le C \int_0^\infty y^{-\delta y} \d y . 
$$
Since $\beta$ can be chosen arbitrary in the interval $(-\zeta,\zeta)$, we obtain 
the desired result.
\end{proof}

Next, for $u \in \ex{\alpha}$ we define  
\begin{equation}\label{def_Xi}
\Xi u(\lambda)=\int_0^{\infty} f_{\overline X}(\lambda x) u(x) \d x. 
\end{equation}

\begin{lemma}\label{lemma8} Let $(\alpha,\rho) \in {\mathcal A}$ and $u \in \ex{\alpha}$. 
\begin{itemize}
\item[(i)] There exists $\epsilon>0$ such that for every $r\ge 0$ the function $\lambda \mapsto \Xi {\mathcal S}_r u(\lambda)$ is analytic in the sector $|\arg(\lambda)|<\pi/\alpha+\epsilon$ where it satisfies
\begin{equation}\label{eqn_Xi_u_bound}
|\Xi {\mathcal S}_r u(\lambda)|\le C \times \min(|\lambda|^{\alpha\rho-1},|\lambda|^{-1}), 
\end{equation}  
for some constant $C=C(u,\alpha,\rho,\epsilon)>0$. 
\item[(ii)] The function $\lambda \in (0,\infty) \mapsto \Pi u(\lambda)$ satisfies 
\begin{equation}\label{eqn_Pi_u_bound}
|\Pi u(\lambda)|\le C \times \min(1,\lambda^{-1}), 
\end{equation}  
for some constant $C=C(u,\alpha,\rho)>0$. 
\end{itemize}
\end{lemma}
\begin{proof}
First we will prove part (i) in the special case $r=0$ (so that ${\mathcal S}_r u=u$).  Recall that the Wiener-Hopf factor $\phi(z)$ (given by \eqref{eq_phi_S2}) is the Laplace transform of $f_{\overline X}(x)$
and that $\phi(z)$ can be extended to a meromorphic function on ${\mathcal R}$ (the Riemann surface of the logarithm function). 
Moreover, according to \eqref{eq_phi_S2} and properties (i) and (ii) of the double sine function 
on page \pageref{properties_S2_function}, we see that the function $\phi(z)$ is analytic in the sector 
$|\arg(z)|<\pi (1/\alpha+\hat \rho)$ and satisfies 
$|\phi(z)| \le  C \min(1,|z|^{-\alpha \rho})$
in any smaller sector $|\arg(z)|<\pi(1/\alpha+\hat\rho)-\epsilon$ for some $C=C(\alpha,\rho,\epsilon)$. 

Let us take $u \in \ex{\alpha}$. 
{Observe that $\Xi u(\lambda)$ is the value of the convolution of $f_{\overline{X}}(\lambda x)$ and $u(-x)$ at $x = 0$ (we extend these functions to $\r$ by setting $f_{\overline{X}}(\lambda x) = u(x) = 0$ when $x < 0$). Both $f_{\overline{X}}(\lambda x)$ and $u(-x)$ are integrable, so the Fourier transform of the convolution of these two functions is the product of the corresponding Fourier transforms: $\lambda^{-1} \phi(-\i z/\lambda)$ and $\laplace u(\i z)$. The estimates discussed above imply that the function $\lambda^{-1} \phi(-\i z/\lambda) \laplace u(\i z)$ is integrable over $z \in \r$, and thus $\Xi u(\lambda)$ can be evaluated as the inverse Fourier transform of $\lambda^{-1} \phi(-\i z/\lambda) \laplace u(\i z)$ at zero. This allows us to write}
\begin{equation*}
\Xi u(\lambda)={}{\frac{1}{2\pi \lambda} \int_{\r} \phi(\i z/\lambda) \laplace u(-\i z) \d z={}}\frac{1}{2\pi \i \lambda} \int_{\i \r} \phi(z/\lambda) \laplace u(-z) \d z. 
\end{equation*}
If $\alpha\le 1$ we choose any $\epsilon \in (0,\pi \hat \rho)$, and if $\alpha>1$ we choose any
$\epsilon \in (0,\frac{\pi}{2} (1-1/\alpha))$.  
Let us define $\beta=\pi/2-\zeta+\epsilon$ (where $\zeta=\frac{\pi}{2}\min(1,1/\alpha)$). 
Applying Lemma \ref{lemma_1}, the upper bound for $\phi$ and the properties of $\laplace u(-z)$ described in Lemma \ref{lemma6}
we can deform the contour of integration in the above equation $\i \r \mapsto \contour$, where 
$\contour=(e^{-\i \beta}\infty,0) \cup (0,e^{\i \beta} \infty)$, and obtain
\begin{equation*}
\Xi u(\lambda)=\frac{1}{2\pi \i \lambda} \int_{\contour} \phi(z/\lambda) \laplace u(-z) \d z.
\end{equation*}
Rewriting this expression as two integrals and changing the variable of integration we arrive at
\begin{align}\label{lemma8_eqn1}
\Xi u(\lambda)&=I_1(\lambda)+I_2(\lambda)\\ 
\nonumber
&=
\frac{e^{\i \beta}}{2\pi \i \lambda} \int_{0}^{\infty} \phi(w e^{\i \beta}/\lambda) \laplace u(-w e^{\i \beta}) \d w
-\frac{e^{-\i \beta}}{2\pi \i \lambda} \int_{0}^{\infty} \phi(w e^{-\i \beta}/\lambda) \laplace u(-w e^{-\i \beta}) \d w.
\end{align}
Since $\phi(z)$ is analytic in the sector $|\arg(z)|<\pi(1/\alpha+\hat \rho)$, the functions $I_1(\lambda)$ and $I_2(\lambda)$ 
are analytic in the sector $|\arg(\lambda)|<\pi (1/\alpha+\hat \rho)-\beta$. Considering separately the two cases $\alpha\le 1$ and $\alpha>1$ we check that $\xi:=\pi \hat \rho-\beta>0$, thus the function $\Xi u(\lambda)$ is analytic in the sector
$|\arg(\lambda)|<\pi/\alpha+\xi$. 

According to Lemma \ref{lemma6}, we can estimate $|\laplace u( -w e^{\pm \i \beta})|\le C \min(1,|w|^{-1})$. 
This result and the bound $|\phi(w e^{\pm \i \beta})| \le C \min(1,|w|^{-\alpha \rho})$ combined with 
Lemma \ref{lemma_Mellin_bound} give us the upper bound \eqref{eqn_Xi_u_bound}. 

The proof of the case $r>0$ follows exactly the same steps, since the function $v_r(x)={\mathcal S}_r u(x)$ also belongs to $\ex{\alpha}$. Lemma~\ref{lemma6} then tells us that the upper bound in \eqref{eqn_Xi_u_bound} is uniform in $r$.  

The proof of part (ii) follows very similar steps. We choose $\epsilon$ and $\beta$ in the same way as above. Let 
$f(z):=\sqrt{2/\pi} \int_0^{\infty}  F(x) e^{-zx} dx$. As we have established in 
formula \eqref{eqn_F_Laplace}, the function $f(z)$ can be analytically continued to a meromorphic function on ${\mathcal R}$. Moreover, $f(z)$ is analytic in the sector $|\arg(z)|<\pi \hat \rho$ and in any smaller sector 
$|\arg(z)|<\pi \hat \rho-\epsilon_2$ it satisfies an upper bound
$|f(z)|\le C \min(1,|z|^{-1-\alpha \rho})$, which follows from  \eqref{S_2_asymptotics} and \eqref{eqn_F_Laplace}.

Note that $e^{- x} F(x)$ and $e^{\lambda x} u(x)$ are in $L^2(\r^+)$. We express $\Pi u$ using Plancherel's theorem for Laplace transform as follows
\begin{align*}
  \Pi u(\lambda) & = \sqrt{2/\pi}\int_0^\infty (e^{-\lambda x} F(\lambda x)) (e^{\lambda x} u(x)) dx = \frac{1}{2 \pi i \lambda} \int_{\lambda + i \r} f(\tfrac{z}{\lambda}) \laplace u(-z) dz .
\end{align*}
The contour of integration is deformed in the same way as in \eqref{lemma8_eqn1} and we obtain
\begin{equation*}
  \Pi u(\lambda)  = \frac{e^{\i \beta}}{2\pi \i \lambda} \int_{0}^{\infty}   f(w e^{\i \beta}/\lambda) \laplace u(-w e^{\i \beta}) \d w
-\frac{e^{-\i \beta}}{2\pi \i \lambda} \int_{0}^{\infty}   f(w e^{-\i \beta}/\lambda) \laplace u(-w e^{-\i \beta}) \d w.
\end{equation*}
Applying Lemma \ref{lemma_Mellin_bound} and the above mentioned estimate 
$|  f(w e^{-\i \beta})|\le C \min(1,|w|^{-1-\alpha \rho})$ we get the desired result 
\eqref{eqn_Pi_u_bound}.
\end{proof}

\begin{lemma}\label{lemma9}
Let $\alpha$ and $\rho$ be as in Theorem~\ref{theorem_2}. Then $\Pi: L^2(\r^+) \mapsto L^2(\r^+)$ is a bounded self-adjoint operator and $ \hat \Pi: \ex{\alpha} \mapsto  L^2(\r^+)$ is a symmetric operator such that $ \hat \Pi \ex{\alpha}$ is dense in $L^2(\r^+)$. 
\end{lemma}
\begin{proof}
The fact that both operators $\Pi$ and $\hat \Pi$ are symmetric follows easily from their definition \eqref{def_operators}. Let us denote by $\mellin$ the Mellin transform operator:
\begin{equation}
\mellin u(z)=\frac{1}{\sqrt{2\pi}}\int_0^{\infty} u(x) x^{-1/2+\i z} \d x, \;\;\; z\in \r. 
\end{equation} 
It is well-known that the Mellin transform operator is an isometry between $L^2(\r^+)$ and $L^2(\r)$.
Since the operator $\Pi$ is a Mellin convolution, Theorem \ref{thm_F_Laplace_Mellin}(ii) and a standard computation leads to identity
$$
\mellin \Pi u(z)=\mellin  F(z) \times \mellin u (-z).  
$$
Formula \eqref{eqn_Mellin_F} and location of poles/zeros of the double sine function imply that $\mellin  F(z)$ has no poles on the real line, and the asymptotic relation \eqref{S_2_asymptotics} tells us that 
$|\mellin  F(z)| \exp((\rho-1/2)\pi |z|) \to 1$ as $z\to \infty$. Thus for $\rho\ge 1/2$ the function $\mellin  F(z)$ is bounded and it induces a bounded multiplication operator on $L^2(\r)$, thus the operator 
$\Pi u=\mellin^{-1} [ \mellin F \times \mellin^{-1}u]$ is a bounded (self-adjoint) operator on $L^2(\r^+)$.  

The fact that $ \hat\Pi$ maps $\ex{\alpha}$ into $L^2(\r^+)$ follows from Lemma \ref{lemma8}(ii). Let us prove that the image 
$ \hat\Pi \ex{\alpha}$ is dense in $L^2(\r^+)$. Assume to the contrary that $ \hat\Pi \ex{\alpha}$ is not dense in $L^2(\r^+)$. Then there must exist a nonzero function $w\in L^2(\r^+)$ such that $\int_0^{\infty} w(x)  \hat\Pi u(x) \d x=0$ for all $u\in \ex{\alpha}$. Consider any nonzero function $u \in \ex{\alpha}$ and define 
$u_a(x)=u(ax)$ for $a>0$. As we have mentioned in the paragraph following Definition \ref{def_set_X}, we have $u_a \in \ex{\alpha}$ for all $a>0$. Let $v(\lambda)= \hat\Pi u(\lambda)$. Then it is clear from the definition of $\hat\Pi$ that $\hat\Pi u_a(\lambda)=a^{-1} v(a^{-1} \lambda)$, so that all functions $v(a^{-1} x)$ belong to $ \hat\Pi \ex{\alpha}$ 
and we have $\int_0^{\infty} w(x) v(a^{-1} x) \d x=0$ for all $a>0$. 
Applying Plancherel's Theorem for Mellin transform we see that 
$$
0=\int_0^{\infty} w(x) v(a^{-1} x) \d x=
\int_{\r} \mellin w(-z) \mellin v(z) a^{1/2+\i z} \d z,  \;\;
{\textnormal{ for all }} \; a>0.  
$$
Thus the Fourier transform of a function $\mellin w(-z) \mellin v(z) \in L^1(\r)$ is identically zero, therefore 
\begin{equation}\label{lemma9_eqn1}
\mellin w(-z) \mellin v(z)=0, \; \; {\textnormal{ for all }} \; z \in \r.
\end{equation}
 Next, due to the fact that $v=\hat \Pi u$ and the estimate 
\eqref{eqn_Pi_u_bound} we see that $\mellin v(z)$ is analytic in a strip $|\im(z)|<1/2$, thus $\mellin v(z)$ is nonzero almost everywhere in this strip. Equation \eqref{lemma9_eqn1} then implies that $\mellin w(-z)=0$ for almost all $z\in \r$, therefore $w \equiv 0$ and we arrive at a contradiction. {(This argument is essentially an application of the $L^2$ version of Wiener's Tauberian theorem for the Mellin transform.)}
\end{proof}

\vspace{0.25cm}
\noindent
{\bf Proof of Theorem \ref{theorem_2}:}
Given the result in Lemma \ref{lemma9}, it remains to establish formulas \eqref{eqn_Pi_hat_Pi}, 
 \eqref{spectral_formula_2} and \eqref{spectral_formula_2_dual}. Our plan is to prove \eqref{spectral_formula_2} 
 using similar method as in the proof of Theorem \ref{thm_main}(ii) and then to derive the remaining statements \eqref{eqn_Pi_hat_Pi} and \eqref{spectral_formula_2_dual} as corollaries of \eqref{spectral_formula_2}. 

When $\rho=1/2$ formula \eqref{spectral_formula_2}  follows at once from \eqref{Q_spectral_formula} and Fubini's Theorem (whose application is justified since the functions $F$ and $\hat F$ are bounded in this case). Thus we will only consider the case when $\rho>1/2$. 

We recall that the operator $\Xi$ is defined by \eqref{def_Xi} and ${\mathcal S}_r$ is the ``shift" operator from 
Lemma~\ref{lemma6}. Let us fix $u \in \ex{\alpha}$. We begin by rewriting formula \eqref{eqn_p_t_P_t}: 
\begin{equation}\label{eqn_Pt_kt}
P_t u(x)=\int_0^{x} k_t(x,z) \d z,
\end{equation} 
where we have defined 
$$
k_t(x,z):=\int_z^{\infty} h_t(x,y,z) u(y) \d y. 
$$
The Laplace transform of $t\mapsto k_t(x,z)$ can be computed by applying Fubini's Theorem to the above formula 
and using \eqref{eqn_H_q_factorization},  \eqref{eqn_Ht_Hq} and \eqref{def_Xi}: 
\begin{align}
\int_0^{\infty} e^{-qt} k_t(x,z)\d t=
\int_z^{\infty} q^{-1} H_q(x,y,z) u(y) \d y=
q^{2/\alpha-1} f_{\underline X}((x-z)q^{1/\alpha}) \Xi {\mathcal S}_z u(q^{1/\alpha})=:K_q(x,z). 
\end{align}

Let us denote $w=(x-z)|q|^{1/\alpha}$. According to Corollary \ref{corollary2} and Lemma \ref{lemma8}, the function 
$q\mapsto  K_q(x,z)$ is analytic in the sector $|\arg(q)|<\pi+\epsilon$ and 
it is bounded in the sector $|\arg(q)|\le \pi$  by
\begin{equation}\label{eqn_Kq_bound1}
|K_q(x,z)|\le  C(u,\alpha,\rho) \times |q|^{2/\alpha-1} \min(w^{\alpha \hat \rho-1},w^{-\alpha-1}) \times \min(|q|^{\rho-1/\alpha},
|q|^{-1/\alpha}). 
\end{equation}
For fixed $z$ and $x$ the upper bound \eqref{eqn_Kq_bound1} implies the estimate: 
\begin{equation}\label{eqn_Kq_bound2}
|K_q(x,z)|\le C(u,\alpha,\rho,x,z)\times \min(1,|q|^{-2}), \;\;\; |\arg(q)|\le \pi. 
\end{equation}
At the same time, using the trivial bounds $\min(w^{\alpha \hat \rho-1},w^{-\alpha-1}) \le w^{\alpha \hat \rho-1}$
and $\min(|q|^{\rho-1/\alpha},|q|^{-1/\alpha})\le q^{\rho-1/\alpha}$ we obtain 
from \eqref{eqn_Kq_bound1} the following result: 
\begin{equation}\label{eqn_Kq_bound3}
|K_q(x,z)|\le C(u,\alpha,\rho)\times |x-z|^{\alpha \hat \rho-1}, \;\;\; |\arg(q)|=\pi. 
\end{equation}

Next, we fix $x,z$ satisfying $0<z<x$, use the upper bound 
\eqref{eqn_Kq_bound2} and express $k_t(x,z)$ as the inverse Laplace transform 
$$
k_t(x,z)=\frac{1}{2\pi \i} \int_{\i \r} K_q(x,z) e^{qt} \d q. 
$$
Next, we use the same upper bound \eqref{eqn_Kq_bound2} and Lemma \ref{lemma_2} and pass to Hankel's contour of integration in the above integral:
$$
k_t(x,z)=-\frac{1}{\pi} \int_0^{\infty} \im[ K_{qe^{\pi \i}} (x,z)] e^{-qt} \d q. 
$$
Combining this result with formula \eqref{eqn_Pt_kt} and applying Fubini's Theorem 
we obtain 
\begin{equation}\label{thm2_proof1}
P_t u(x)=-\frac{1}{\pi} \int_0^{\infty} 
\left[ \int_0^{x}  \im[ K_{qe^{\pi \i}} (x,z)] \d z \right] e^{-qt} \d q. 
\end{equation} 
The use of Fubini's Theorem in the previous step is justified since the function 
$(z,q) \mapsto K_{qe^{\pi \i}} (x,z) e^{-qt}$ is absolutely integrable on $(0,x) \times (0,\infty)$ -- this follows easily from
the upper bound \eqref{eqn_Kq_bound3}.

Next, we rewrite the integral in  square brackets in \eqref{thm2_proof1} as follows
\begin{align}\label{thm2_proof2}
& \int_0^{x}  \im[ K_{qe^{\pi \i}} (x,z)] \d z \\ \nonumber
&\qquad =
 -q^{2/\alpha-1} \int_0^{x}  \int_z^{\infty} 
 \im \left[e^{2\pi \i/\alpha} f_{\underline X}((x-z)q^{1/\alpha} e^{\pi \i/\alpha})
 f_{\overline X}((y-z)q^{1/\alpha} e^{\pi \i/\alpha}) u(y) \right] \d y \d z \\ \nonumber
 &\qquad =
 -q^{2/\alpha-1}   \int_0^{\infty} \int_0^{\min(x,y)}
 \im \left[e^{2\pi \i/\alpha} f_{\underline X}((x-z)q^{1/\alpha} e^{\pi \i/\alpha})
 f_{\overline X}((y-z)q^{1/\alpha} e^{\pi \i/\alpha})  \right] \d z u(y) \d y
 \\ \nonumber
 &\qquad = - \frac{2}{\alpha} q^{1/\alpha-1}  \int_0^{\infty}  F(xq^{1/\alpha}) \hat F(yq^{1/\alpha}) u(y) \d y
 =- \frac{2}{\alpha}  q^{1/\alpha-1} F(x q^{1/\alpha}) \hat \Pi u(q^{1/\alpha}). 
\end{align} 
In the second step we have again applied Fubini's Theorem: to justify its use note that according to 
Corollary \ref{corollary2} we have
$$
|f_{\underline X}((x-z)q^{1/\alpha} e^{\pi \i/\alpha})
 f_{\overline X}((y-z)q^{1/\alpha} e^{\pi \i/\alpha}) u(y)|\le
 C \times |x-z|^{\alpha \hat \rho-1} \times |y-z|^{\alpha \rho-1} e^{y} \times |u(y)|
$$
for some constant $C=C(\alpha,\rho,q)$, and the function in the right-hand side of the above equation 
is integrable over the region $\{(z,y) \in \r^2 : 0<z<\min(x,y)\}$. In the third step of \eqref{thm2_proof2} we have applied 
Corollary \ref{corollary_lucky_integral}. Combining formulas \eqref{thm2_proof1} and \eqref{thm2_proof2} and 
changing the variable of integration $q=\lambda^{\alpha}$ we obtain the desired result 
\eqref{spectral_formula_2}.

Formula \eqref{eqn_Pi_hat_Pi} follows from \eqref{spectral_formula_2} by taking the limit $t\to 0^+$: Then $P_tu(x) \to u(x)$ for all $x>0$ (since $u$ is continuous and a L\'evy process $X$ killed on the first exit from $(0,\infty)$ is continuous in probability), while the right-hand side of \eqref{spectral_formula_2} converges to $\Pi \hat \Pi u$. 

Let us now prove formula \eqref{spectral_formula_2_dual}. We denote by $(u,v)$ the inner product in $L^2(\r^+)$. We start with $v,w \in \ex{\alpha}$ and use 
formulas \eqref{eqn_Pi_hat_Pi} and \eqref{spectral_formula_2} and the fact that operators $\Pi$ and $\hat \Pi$ are symmetric to compute
\begin{equation}\label{L2_computation1}
(\hat P_t \hat \Pi w, v)=(\hat \Pi w, P_t v)=(\hat \Pi w, \Pi e^{-\lambda^{\alpha}t} \hat \Pi v)
=(\Pi \hat \Pi w, e^{-\lambda^{\alpha}t} \hat \Pi v)
=( w, e^{-\lambda^{\alpha}t} \hat \Pi v)=(\hat \Pi e^{-t \lambda^{\alpha}} w, v). 
\end{equation}
In the last step we have used the fact that $e^{-t \lambda^{\alpha}} w \in \ex{\alpha}$. Since 
\eqref{L2_computation1} is valid for all $v \in \ex{\alpha}$ and $\ex{\alpha}$ is dense in $L^2(\r^+)$ we conclude that $\hat P_t \hat \Pi w=\hat \Pi e^{-t \lambda^{\alpha}} w$. Therefore, $\Pi \hat P_t \hat \Pi w=\Pi \hat \Pi e^{-t \lambda^{\alpha}} w = e^{-t \lambda^{\alpha}} w$, as desired.
\qed

\vspace{0.25cm}
\noindent
{\bf Proof of Corollary \ref{corollary1}:} Part (ii) of Corollary \ref{corollary1} was established above (see the discussion following \eqref{L2_computation1}). To prove part (i), we start with $u \in \ex{\alpha}$, $v \in L^2(\r^+)$ and denote 
$\hat \Pi u=w$. According to \eqref{eqn_Pi_hat_Pi} we have $\Pi w=u$. 
We use formula \eqref{spectral_formula_2_dual} and compute 
\begin{equation}\label{corrollary1_proof1}
(P_t \Pi v, w)=(P_t \Pi v, \hat \Pi u)=(v, \Pi \hat P_t \hat \Pi u)=
(v,e^{-t \lambda^{\alpha}}u)=
(e^{-t \lambda^{\alpha}}v, \Pi w)=(\Pi e^{-t \lambda^{\alpha}}v,  w). 
\end{equation}
Since $w$ ranges over $\hat \Pi \ex{\alpha}$, and the latter is a dense subset of $L^2(\r^+)$ due to 
Theorem \ref{theorem_2}(i), we see that \eqref{corrollary1_proof1} implies that for all $v \in L^2(\r^+)$ 
we have $P_t \Pi v=\Pi e^{-t \lambda^{\alpha}}v$. Application of Fubini's Theorem (note that $u_\lambda(x) = F(\lambda x)$ is bounded when $\rho \ge 1/2$) gives
\[
 P_t \Pi v(x) = \int_0^\infty \int_0^\infty p_t(x, y) F(\lambda y) v(\lambda) \d \lambda \d y = \int_0^\infty P_t u_\lambda(x) v(\lambda) \d \lambda .
\]
Since $v \in L^2(\r^+)$ is arbitrary, we have $P_t u_\lambda(x) = e^{-t \lambda^\alpha} u_\lambda(x)$ for almost all $x, \lambda > 0$. By continuity, this relation holds for all $x, \lambda > 0$, as desired.
\qed

\section{Doney classes}\label{Section_Doney_classes}

Doney \cite{Doney1987} has introduced classes of stable processes for which the Wiener-Hopf factors can be computed explicitly in terms 
of q-Pochhammer symbols. Let $k,l\in {\mathbb Z}$ and let us define by ${\mathcal C}_{k,l}$ the class of stable processes with parameters $(\alpha,\rho)$ satisfying
\begin{equation}\label{def_Doney_class}
\alpha\rho=l-k\alpha.
\end{equation}
These classes include the spectrally-positive $({\mathcal C}_{0,1})$ and the spectrally-negative processes $({\mathcal C}_{-1,-1})$. 
Note that $X \in {\mathcal C}_{k,l}$ if and only if $\hat X \in {\mathcal C}_{-k-1,-l}$. 
Let us define the q-Pochhammer symbol 
\begin{align*}
(a;q)_n=
\begin{cases}
\prod\limits_{j=0}^{n-1} (1-aq^j), \;\;\; n>0,\\
\prod\limits_{j=1}^{|n|} (1-aq^{-j})^{-1}, \;\;\; n<0,
\end{cases}
\end{align*}
and $(a;q)_0=1$. 

The main result of Doney \cite{Doney1987} is the following formula for the Wiener-Hopf factor: for a process $X\in {\mathcal C}_{k,l}$ we have 
 \begin{equation}
 \phi(z)=\frac{\left((-1)^{l+1}z^{\alpha}e^{\pi \i \alpha (1-k)} ;q\right)_k}
{\left((-1)^{k-1} z e^{\pi \i (l-1)/\alpha} ;\tilde q\right)_l}, 
 \end{equation}
where $q:=e^{2\pi \i \alpha}$ and $\tilde q:=e^{-\frac{2\pi \i}{\alpha}}$. 
 We note that this formula follows from our general result (Theorem \ref{thm_phi_S2}) and formula \eqref{S2_Doney}. It is not surprising that all expressions involving the eigenfunctions $F(x)$ also simplify considerably for Doney classes ${\mathcal C}_{k,l}$. 
 \begin{proposition}\label{prop_Doney}
 Assume that $X \in {\mathcal C}_{k,l}$. 
 \begin{itemize} 
 \item[(i)] The function $G(x)$ (which is related to $F(x)$
 through \eqref{def_F_eigenfunction}) is  given by 
\begin{equation}\label{def_g2}
G(x)=\int_0^{\infty} e^{-z x} 
 z^{\alpha} \frac{\left((-1)^{l}z^{\alpha}e^{\pi \i \alpha (k+3)} ;q\right)_{-k-2}}
 {\left((-1)^{k+1} z e^{-\pi \i l/\alpha} ;\tilde q\right)_{-l+1}} \d z. 
\end{equation}
\item[(ii)] For $\re(z)>\max(0,\cos(\pi \rho))$ 
\begin{equation}\label{eqn_F_Laplace2}
\int_0^{\infty} e^{-zx} F(x) \d x=
\frac{\sqrt{\alpha}}{2} S_2(\alpha \rho) 
\frac{\left((-1)^{l}z^{\alpha}e^{\pi \i \alpha (k+2)} ;q\right)_{-k-1}}
{\left((-1)^{k} z e^{-\pi \i l/\alpha} ;\tilde q\right)_{-l+1}}.
\end{equation}
\item[(iii)] Assume that $\rho \ge 1/2$ and $\re(z) \in (-\alpha \hat \rho,0)$. If $l>0$ then
\begin{equation}\label{Mellin_F_Ckl_+}
\int_0^{\infty} x^{z-1} F(x) \d x=\frac{1}{2} (-1)^{(k+1)l} \Gamma(z) 
 \prod\limits_{j=1}^{k+1} 2 \sin(\pi (z+(j-1)\alpha))
 \prod\limits_{j=1}^l \frac{1}{2\sin(\pi (z-j)/\alpha)}, 
\end{equation}
while if $l<0$ we have
\begin{equation}\label{Mellin_F_Ckl_-}
\int_0^{\infty} x^{z-1} F(x) \d x=\frac{1}{2} (-1)^{(k+1)l} \Gamma(z) 
 \prod\limits_{j=1}^{|l|} 2 \sin(\pi (z+j-1)/\alpha)
 \prod\limits_{j=1}^{|k+1|} \frac{1}{2\sin(\pi (z-j\alpha))}, 
\end{equation}
\end{itemize}
\end{proposition}
\begin{proof}
The proof follows immediately from \eqref{def_g}, \eqref{eqn_F_Laplace}, \eqref{eqn_Mellin_F} and formulas 
\eqref{S2_Doney} and \eqref{S2_Doney2}. 
\end{proof}

\subsection{Spectrally one-sided processes}\label{section_one_sided}

The spectrally-negative (spectrally-positive) processes  belong to classes ${\mathcal C}_{0,1}$ (respectively, 
${\mathcal C}_{-1,-1}$). All the formulas given in Proposition \ref{prop_Doney} for general Doney classes 
${\mathcal C}_{k,l}$ remain valid, except that in the spectrally-negative case 
$\alpha \rho=1$ we should take $G(x)\equiv 0$. Let us consider this case in more detail. 

In the spectrally-negative case it is well-known that the random variable ${\overline X}_{\ee(1)}$ has exponential distribution 
with parameter one, so that 
$$
f_{\overline X}(x)=e^{-x}, \;\;\; x>0. 
$$
Following the same analysis as before (or simply using formula \eqref{p_piialpha_F}), we would obtain
the co-eigenfunction 
$$
\hat F(x)=e^{-x \cos(\pi/\alpha)} \sin(x \sin(\pi/\alpha)). 
$$ 
The eigenfunction is given by \eqref{def_F_eigenfunction} and \eqref{def_g2}: 
\begin{align*} F(x)&=e^{x\cos(\pi/\alpha)}\sin(x\sin(\pi/\alpha)+\pi (2-\alpha)/(2\alpha))\\
&+\frac{\alpha}{2\pi} \sin(\pi \alpha)
\int_0^{\infty} \frac{e^{- ux} u^{\alpha} \d u}{1+2\cos(\pi \alpha)u^{\alpha} + u^{2\alpha}}.
\end{align*}

\section{Conclusion}\label{section_conclusion}

Our methods for establishing Theorems \ref{thm_main}, \ref{theorem_2} and \ref{thm_F_Laplace_Mellin} are fundamentally based on
the Wiener-Hopf factorization (Theorem \ref{theorem_WHF}), which comes from the fluctuation theory of L\'evy processes. We hope that the methods developed in this paper will help to study spectral properties of semigroups of more general L\'evy processes on a half-line (see \cite{Kwasnicki2013} for some results in this direction). Of course, the case of stable processes is a very special one, as stable processes lie in the intersection of the class of L\'evy processes and the class of self-similar Markov processes. We would like to mention that the spectral properties of the semigroup of general positive self-similar Markov processes (which include stable processes killed on the first exit from $(0,\infty)$) are studied in \cite{Patie_Savov_Zhao} using a different approach (based on Lamperti transformation and Mellin transform techniques). 

Finally, we would like to emphasize the importance of the double sine function $S_2(z;\alpha)$ in the study of stable processes. This function appears in the definition \eqref{def_F_eigenfunction} of the eigenfunctions, in formulas  \eqref{eqn_F_Laplace} and \eqref{eqn_Mellin_F} that give the Laplace and Mellin transform of the eigenfunctions and in formula \eqref{eq_phi_S2}, which gives the Wiener-Hopf factors of stable processes. So far, apart from some number-theoretic applications \cite{Koyama2007204,Tanaka},
 this curious special function has appeared  mostly in the Physics literature \cite{Fock,Ponsot_Teschner,Ruijsenaars2,Volkov}, where it is used in studying quantum topology and cluster algebras. It is an interesting question whether this appearance of the double sine function
 in our study of stable processes is simply a coincidence or there is indeed a deeper connection between stable processes on the half-line and quantum topology and/or cluster algebras.

\section*{Acknowledgements}
Research of A. Kuznetsov was supported by the Natural Sciences and Engineering Research Council of Canada. 
M.~Kwa\'snicki was supported by the Polish National Science Centre (NCN) grant no. 2011/03/D/ST1/00311.
The authors would like to thank Pierre Patie and Mladen Savov for insightful discussions and two anonymous referees for careful reading of the paper and for many helpful suggestions.



\begin{appendices}

    \setcounter{proposition}{0}
    \renewcommand{\theproposition}{\Alph{section}\arabic{proposition}}
    \setcounter{theorem}{0}
    \renewcommand{\thetheorem}{\Alph{section}\arabic{theorem}}
    \numberwithin{equation}{section}

\section{The double sine function}\label{AppendixA}
    \renewcommand{\theequation}{\Alph{section}.\arabic{equation}}

In a series of papers Koyama and Kurokawa (see \cite{Kurokawa,Koyama2007204}, for example) have studied the multiple sine functions $S_r(z;{\mathbf w})$ where $z\in \c$, $r\in {\mathbb N}$ 
and ${\mathbf w}=(w_1,\dots,w_r)$, $w_i \in \c \setminus (-\infty,0]$. These are homogeneous functions (in the sense that $S_r(cz;c{\mathbf w})=S_r(z;{\mathbf w})$ for any $c>0$) that satisfy certain functional equations. Considering the case $r=2$, 
the function  $S_2(z;{\mathbf w})$ satisfies 
$$
S_2(z+w_i;{\mathbf w})=\frac{S_2(z;{\mathbf w})}{2\sin(\pi z /w_i)}, \;\;\; i \in \{1,2\}. 
$$
It can be shown \cite{Bult} that $S_2(z;{\mathbf w})$ is a unique meromorphic function of $z$ which satisfies the above functional equations and has value one at $z=(w_1+w_2)/2$. Due to the homogeneity of $S_2(z;(w_1,w_2))$ we can set $w_1=1$, and from now on we will write simply 
$S_2(z;\alpha)=S_2(z;(1,\alpha))$. Note that the same homogeneity property implies an identity 
\begin{equation}\label{S_2_modular}
S_2(z;\alpha)=S_2(z/\alpha;1/\alpha).
\end{equation}

As the double sine function is becoming better known and more widely used in Mathematics and Physics literature, 
it seems that most authors introduce a new notation for this object, which makes it very hard to navigate the literature and to find relevant results. Below we provide a short summary of various notations used for the double sine function $S_2(z;\alpha)$:
\begin{itemize}
\item[(i)]  The double sine function $S_b(z)$ of Ponsot and Teschner \cite{Ponsot_Teschner}:
$$
S_2(z;\alpha)=S_{\sqrt{\alpha}}(z/\sqrt{\alpha})^{-1}.
$$
\item[(ii)] The hyperbolic gamma function $G_{{\textnormal{hyp}}}(w_1,w_2;z)$ of Ruijsenaars \cite{Ruijsenaars2,Ruijsenaars}: 
$$
S_2(z;\alpha)=G_{{\textnormal{hyp}}}(1,\alpha;\i (1/2+\alpha/2-z))^{-1}.
$$
\item[(iii)] The hyperbolic gamma function $\Gamma_h(z;w_1,w_2)$ of van de Bult \cite{Bult}: 
$$
S_2(z;\alpha)=\Gamma_h(z;1,\alpha)^{-1}.
$$ 
\item[(iv)] The double gamma function $G(z;\alpha)$ of Barnes \cite{Barnes1899}:
\begin{equation}\label{eqn_S2_Barnes_Gamma}
S_2(z;\alpha)=(2\pi)^{(1+\alpha)/2-z} \frac{G(z;\alpha)}{G(1+\alpha-z;\alpha)}. 
\end{equation}
\item[(iv)] The $\gamma$-function (see \cite{Volkov} and the references therein): assume that $\im(\alpha)>0$ and define $q:=e^{2\pi \i \alpha}$, $\tilde q:=e^{-2\pi \i/\alpha}$ and 
\begin{equation}\label{def_gamma_function}
\gamma(z;\alpha):=\frac{(q e^{-2\i \pi z};q)_{\infty}}{(e^{-2 \pi \i z /\alpha};\tilde q)_{\infty}},
\end{equation}
where $(a,q)_{\infty}:=\prod_{k\ge 0} (1-aq^k)$ is the q-Pochhammer symbol. The $\gamma$-function is related to the double sine function through the identity
\begin{equation}\label{eqn_S_2_gamma}
S_2(z;\alpha)=e^{-\pi \i (2z-1-\alpha)^2/(8\alpha)+\pi \i (\alpha+1/\alpha)/24}  \times  \frac{1}{\gamma(z;\alpha)}. 
\end{equation}
\item[(iv)] The quantum dilogarithm function $\Phi^{\hbar}(z)$ (see \cite{Fock} and the references therein) is related to the double sine
 function  $S_2(z;\alpha)$ through the above formula \eqref{eqn_S_2_gamma} and the identity
$$
\gamma(z;\alpha)=\Phi^{\alpha}(\pi \i (1+\alpha-2z))^{-1}. 
$$
\end{itemize}

 \begin{figure}
\centering
\includegraphics[height =8cm]{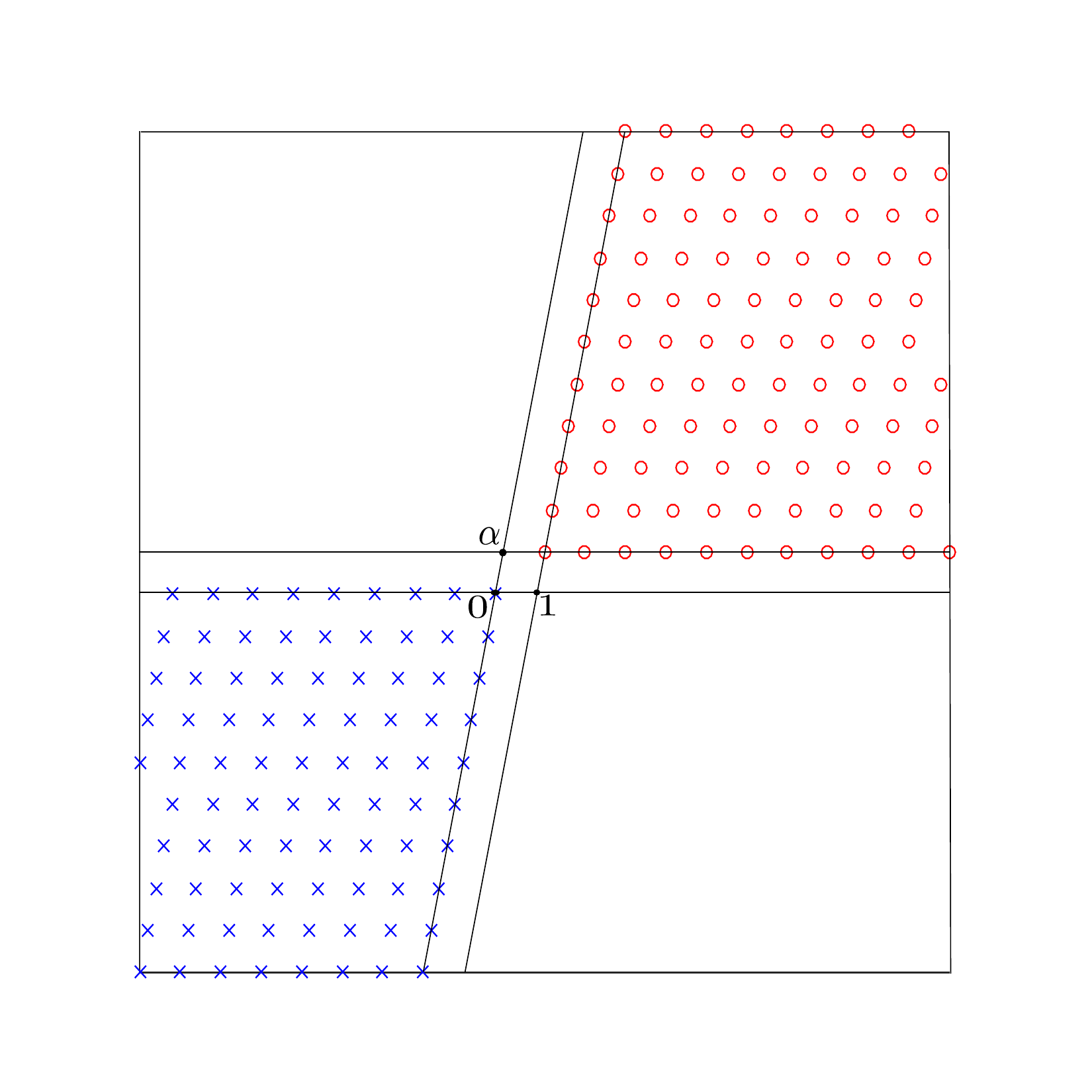} 
\caption{The roots (blue crosses) and the poles (red circles) of the double sine function $S_2(z;\alpha)$.}
\label{fig_2}
\end{figure}
 
For fixed $\alpha \in \c \setminus (-\infty,0]$, 
the double sine function is meromorphic in $z$-variable and it admits the product representation of the form
\begin{equation}\label{eqn_S2_Weierstrass}
S_2(z;\alpha)=e^{A(\alpha)+B(\alpha)z+C(\alpha)z^2} \frac{z}{z-1-\alpha}
\prod\limits_{m\ge 0} \prod\limits_{n\ge 0} {}^{'}
\frac{P(-z/(m+\alpha n))}{P(z/(m+1+\alpha(n+1)))}
\end{equation}
where $P(z):=(1-z)\exp(z+z^2/2)$ and the prime in the second product means that the term corresponding to $m=n=0$ is omitted.
The explicit expressions for $A(\alpha)$, $B(\alpha)$ and $C(\alpha)$ can be found in 
\cite[Theorem 6]{Tanaka}. 
It is clear from \eqref{eqn_S2_Weierstrass} that $S_2(z;\alpha)$ has zeros at 
points $\{-m-\alpha n \; : \; m\ge0, n\ge 0\}$ and poles at points $\{m+\alpha n \; : \; m\ge 1, n\ge 1\}$. Moreover, the pole 
at $z=\alpha+1$ is simple (other poles may have multiplicity greater than one when $\alpha$ is rational). Note that if
$\im(\alpha)>0$ and if we look at the lattice $\{m+\alpha n, \; m, n\in {\mathbb Z}\}$, then the poles (zeros) 
of $S_2(z;\alpha)$ lie in the first (respectively, third) quadrant of this lattice (see Figure \ref{fig_2}). 

Besides the value $S_2(1/2+\alpha/2;\alpha)=1$, the following values are known explicitly \cite{Koyama2007204}: 
\begin{equation}\label{S2_special_values}
S_2(1;\alpha)=1/S_2(\alpha;\alpha)=\sqrt{\alpha}, \;\;\; 
S_2(1/2;\alpha)=S_2(\alpha/2;\alpha)=\sqrt{2}. 
\end{equation} 
The above result \eqref{eqn_S2_Barnes_Gamma} implies a very useful reflection formula
\begin{equation}\label{S2_reflection_formula}
S_2(z;\alpha)S_2(1+\alpha-z;\alpha)=1.
\end{equation}

Let us denote $q=e^{2\pi \i \alpha}$ and $\tilde q=e^{-\frac{2\pi \i}{\alpha}}$. 
Using \eqref{def_gamma_function}, \eqref{eqn_S_2_gamma}, \eqref{S2_reflection_formula} and the identity $(a;q)_{\infty}=(aq^n;q)_{\infty} (a;q)_n$ we can derive a useful result
\begin{align}\label{S2_Doney}
&S_2(1/2+\alpha/2+(m-n\alpha)/2 + z;\alpha)S_2(1/2+\alpha/2+(m-n\alpha)/2 - z;\alpha)
 \\ & \qquad \qquad \nonumber
=
e^{-\pi \i (m-\alpha n) z/\alpha}\frac{((-1)^{m-1}e^{-2\pi \i z+\pi \i \alpha(1-n)};q)_n}
{((-1)^{n-1} e^{-2\pi \i z/\alpha+\pi \i (m-1)/\alpha};\tilde q)_m}. 
\end{align}
When $m,n \in {\mathbb Z}^+$, we can apply repeatedly the functional equations \eqref{S2_two_functional_eqns} and prove the following result
\begin{equation}\label{S2_Doney2}
\frac{S_2(z;\alpha)}{S_2(z+m-n\alpha;\alpha)}=(-1)^{mn} \prod\limits_{j=1}^m 2\sin(\pi (z+j-1)/\alpha) 
 \prod\limits_{j=1}^n \frac{1}{2 \sin(\pi(z-j \alpha))}.
\end{equation}
A similar expression when $-m, -n \in {\mathbb Z}^+$ can be easily obtained from the above formula by changing variables 
$z+m-n\alpha=w$ (or by using transformation \eqref{S_2_modular} or the reflection formula \eqref{S2_reflection_formula}). 

An asymptotic behavior of the double sine function as the imaginary part of the argument increases is described by the following equation 
\begin{equation}\label{S2_asymptotic}
  |S_2(c+\i y;\alpha)| =e^{\pi |y| (1+\alpha-2c)/(2\alpha)}(1+o(1)), \;\;\; y\to \infty. 
\end{equation}
See formula 135 in \cite{Ponsot_Teschner}. It is easy to prove that this result holds  uniformly for $c$ on compact subsets of $\r$. We note that this result implies \eqref{S_2_asymptotics}.

Finally, the ``b-beta integral" of Ponsot-Teschner (Lemma 15 in \cite{Ponsot_Teschner}, see also formula 7 in \cite{Volkov}) implies the following result: For $\alpha>0$, $0<b<(1+\alpha)/2$ and $s\in (-b,b)$ we have
\begin{equation}\label{eqn_tau_binomial}
\int_{0}^{\infty} x^{s-1} |S_2(1/2+\alpha/2+b + \i \alpha \ln(x)/(2\pi);\alpha)|^2 \d x=\frac{2\pi}{\sqrt{\alpha}} 
\frac{S_2(2b;\alpha)}{S_2(b+s;\alpha)S_2(b-s;\alpha)}. 
\end{equation}

\end{appendices}

\end{document}